\documentclass[english,11pt]{article}
\usepackage[utf8]{inputenc}
\pdfoutput=1
\usepackage{amsmath}
\usepackage{tikz}
\usepackage{amsthm}
\usepackage{amssymb}
\usepackage{mathdots}
\usepackage{mathtools}
\usepackage{booktabs}
\usepackage[toc,page]{appendix}
\usepackage{amsxtra}
\usepackage{amsbsy}
\usepackage{verbatim}
\usepackage{comment}
\usepackage{lipsum}
\usepackage{adjustbox}
\usepackage{mathrsfs}
\usepackage{tabularx,ragged2e,booktabs,caption}
\usepackage{leftidx}
\usepackage{longtable}
\usepackage{lipsum}
\usepackage{enumerate}
\usepackage{blindtext}
\usepackage{url}
\usepackage{bbm}
\usepackage{euscript}
\usepackage{pifont}
\RequirePackage{tikz-cd}
\usepackage[T1]{fontenc}
\usepackage{babel}
\usepackage{csquotes}
\usepackage{hyperref}
\usepackage{wasysym}
\usepackage{fancybox}
\usepackage{scalerel,stackengine}
\usepackage{titlesec}
\titleformat{\section}
  {\centering\normalfont\normalsize\bfseries}{\thesection.}{1em}{}
\titleformat{\subsection}
  {\centering\normalfont\normalsize\bfseries}{\thesubsection.}{1em}{}
\hypersetup{
    colorlinks=true,
    citecolor=black,
    linkcolor=blue,
    filecolor=magenta,
    urlcolor=cyan,
}
\stackMath
\newcommand\reallywidehat[1]{%
\savestack{\tmpbox}{\stretchto{%
  \scaleto{%
    \scalerel*[\widthof{\ensuremath{#1}}]{\kern-.6pt\bigwedge\kern-.6pt}%
    {\rule[-\textheight/2]{1ex}{\textheight}}
  }{\textheight}%
}{0.5ex}}%
\stackon[1pt]{#1}{\tmpbox}%
}

\usepackage{graphicx}
\usepackage{epstopdf}

    \usepackage[
    top    = 2.60cm,
    bottom = 2.60cm,
    left   = 2.9 cm,
    right  = 2.9 cm]{geometry}

\newtheorem{thm}{Theorem}[section]
\newtheorem{lem}[thm]{Lemma}
\newtheorem{cor}[thm]{Corollary}
\newtheorem{prop}[thm]{Proposition}

\newtheorem*{conjecture*}{Conjecture}
\newtheorem*{thm*}{Theorem}

\theoremstyle{definition}
\newtheorem{define}[thm]{Definition}
\newtheorem{remark}{Remark}
\newtheorem{example}[thm]{Example}

\newcommand{\1}{\mathbf{1}}
\newcommand{\N}{\mathbb{N}}
\newcommand{\Z}{\mathbb{Z}}
\newcommand{\Q}{\mathbb{Q}}
\newcommand{\R}{\mathbb{R}}
\newcommand{\C}{\mathbb{C}}

\newcommand{\K}{\Bbbk}

\newcommand{\BasedBim}{\operatorname{BasedBim}}
\newcommand{\BasedAlg}{\operatorname{BasedAlg}}
\newcommand{\grBasedBim}{\operatorname{grBasedBim}}
\newcommand{\filBasedBim}{\operatorname{filBasedBim}}

\newcommand{\Bim}{\operatorname{Bim}}
\newcommand{\PSU}{\operatorname{PSU}}
\newcommand{\id}{\operatorname{id}}
\newcommand{\Hom}{\operatorname{Hom}}

\newcommand{\Coker}{\operatorname{coker}}
\newcommand{\ev}{\operatorname{ev}}
\newcommand{\coev}{\operatorname{coev}}
\newcommand{\Irr}{\operatorname{Irr}}

\newcommand{\Mat}{\operatorname{Mat}}
\newcommand{\Stab}{\operatorname{Stab}}

\newcommand{\Mod}{\operatorname{Mod}}
\newcommand{\modd}{\operatorname{mod}}
\newcommand{\End}{\operatorname{End}}

\newcommand{\KQ}{\operatorname{\K Q}}

\newcommand{\Id}{\operatorname{Id}}
\newcommand{\VecC}{\operatorname{Vec}}

\newcommand{\Aut}{\operatorname{Aut}}
\newcommand{\Ind}{\operatorname{Ind}}
\newcommand{\Fib}{\operatorname{Fib}}
\newcommand{\Rep}{\operatorname{Rep}}
\newcommand{\Fun}{\operatorname{Fun}}

\newcommand{\coRep}{\operatorname{coRep}}
\newcommand{\coEnd}{\operatorname{coEnd}}

\newcommand{\Set}{\operatorname{Set}}
\usepackage{graphicx} 

\title{Actions of Fusion Categories on Path Algebras}
\author{Alexander Betz}
\date{}

\begin{document}

\maketitle
\begin{abstract}
    In this article, we introduce the notion of a based action of a fusion category on an algebra. We will build some general theory to motivate our interest in based actions, and then apply this theory to understand based actions of fusion categories on path algebras $\KQ$. Our results demonstrate that a separable idempotent split based action of a fusion category $C$ on a path algebra $\KQ$ can be characterized in terms of $C$ module categories and their associated module endofunctors. As a specific application, we fully classify separable idempotent split based actions of the family of fusion categories $\PSU(2)_{p-2}$ on path algebras up to conjugacy. 
    \end{abstract}
\tableofcontents

\section{Introduction}

 Fusion categories are rich mathematical objects having applications to mathematics and physics \protect\cite{ENO,Tur10,Jon97}, generalizing finite groups and their representation categories. Classically, finite symmetries of associative algebras are characterized by group actions. A natural extension of this concept, Hopf algebra actions, are one characterization of finite quantum symmetries \protect\cite{BD,EW14,KW16,EKW21}. However, Hopf algebra actions produce categorical symmetries, which in turn give rise to another generalization of quantum symmetries: fusion category actions on algebras, defined as monoidal functors from $C$ into $\Bim(A)$. This motivates us to study the quantum symmetries of categorical actions.

  Much of the work on fusion categories actions has been carried in the context of $C^*$ algebras \protect\cite{HP17,HHP20,EJ24,CJHP24}. Considering these concepts from an algebraic point of view leads us to examine the case of an associative algebra $A$. One natural program is to look for fusion category action analogs of results for Hopf algebra actions. However, when comparing fusion category actions and Hopf algebra actions, there are some differences. Noticeably, Hopf algebra actions appear to have more structure than fusion category actions. For example, Hopf algebra actions have a notion of fixed points of an action. If $A$ is a graded/filtered algebra we can extend this concept of fixed points to a graded/filtration preserving action of $H$ on $A$. Given a group $G$ acting on $A$ we associate a bimodule $A_g$ to each group element. Then the fixed points of this bimodule are points such that the left and right action of $x$ on $1_A$ are equivalent. We can define a grading preserving action by requiring our the action of $g(A_n)\subset A_n$ where $A_n$ is the $n^{th}$ graded component of $A$. A filtration preserving action is defined similarly. This definition seemlessly translates when we change our persepctive to bimodules. The notion also extends naturally to Hopf algebras as well. An action of $C$ on $A$ doesn't  capture this notion of fixed points, and consequently we don't have a notion of filtration preserving or grading preserving actions. To produce a categorical analog for these concepts, we introduce the definition of a based action of $C$ on $A$.

We begin by formulating a 2-category of based actions of fusion categories on algebras and establishing different notions of equivalence. Subsequently, we establish for a semisimple Hopf algebra $H$, based actions of $\coRep(H)$ on $A$ with some additional properties are equivalent to actions of $H$ on $A$.  Furthermore, we apply our general theory to the problem of classifying based actions of $C$ on the path algebra $\KQ$ where $\K$ is an algebraically closed field of characteristic zero. It's a well-known result that all finite dimensional algebras are Morita equivalent to a quotient of a path algebra, motivating our desire to better understand symmetries in this context.

There has already been progress towards classifying quantum symmetries of path algebras. In \protect\cite{EKW21} it was shown that actions of Hopf algebra on path algebras that preserve the filtration are classified by tensor algebras internal to $\Rep(H)$. Tensor algebras also appear in the literature to classify bimodules of Hopf algebra actions on path algebras \protect\cite{KO}. In this paper, we translate the language of tensor algebras into our framework of based actions, generalizing some semisimple Hopf algebra results in \protect\cite{EKW21} to fusion categories. Specifically, separable idempotent split based actions of fusion categories on $\KQ$ that preserve the filtration are classified by a $C$ module category and a $C$ module endofunctors in $\End_C(M)$ up to an equivalence we define later. Notably, our framework of based actions are more general than filtered actions of a fusion category on a path algebra, as there is no restriction placed on the half braiding.

We show that separable idempotent split based actions of fusion categories on path algebras are classified by a semisimple module categories plus a conjugacy class of module endofunctors in $\Ind(\End_C(M))$. 
Semisimple module categories of a fusion category $C$ can be understood in terms of a Morita class of algebras internal to $C$ \protect\cite{Ost03}. In particular, in Theorem \protect\ref{thm:TE BimEnd} we produce the following result, 
\begin{thm}
There is an equivalence of monoidal categories between $\BasedBim(\KQ)_{I_v}$ and  $\End_{(\widetilde{Q},m,i)}(\VecC(M))$, where $M$ is the semisimple category equivalent to $\Mod(\KQ_0)$.
\end{thm}

Here $\BasedBim(\KQ)_{I_v}$ represents the corner of themonoidal category based bimodules of the path algebra where the vertex projections act as the unit and $\KQ$  and $\End_{(\widetilde{Q},m,i)}(\VecC(M))$ is the collection of endofunctors that commute with a monad $\widetilde{Q}$ and are compatible with the algebra structure. This category is a subcategory of $\End_{\widetilde{Q}}(\VecC(M))$ which we have a nice classification of monoidal functors from $C$ into $\End_{\widetilde{Q}}(\VecC(M))$ when the connected components of $Q$ are strongly connected.

\begin{thm}
Let $Q$ be a quiver with strongly connected components. Let $C$ be a fusion category and fix a semisimple $C$ module category structure on $M$. If there exists a monoidal functor $F:C\rightarrow \End_{\widetilde{Q}}(\VecC(M))$ then there are up to monoidal natural isomorphism there are either $1$ or $\infty$ such functors.
\end{thm}

\begin{cor} 
Let $Q$ be a quiver with strongly connected components. Then, for a fixed module category structure on $M$ if there exists a monoidal functor $F:C\rightarrow \End_{(\widetilde{Q},m,i)}(\VecC(M))$ there are up to monoidal natural isomorphism there are either at most $1$ or at most $\infty$ such functors.
\end{cor}

 We use this theory to build families of filtered/graded actions of a fusion category $C$ on a path algebra $\KQ$. Then, for the fusion category $\PSU(2)_{p-2}$, we produce the following results about separable based actions of $\PSU(2)_{p-2}$ on path algebras.

 \begin{thm}
   Every separable based action of $\PSU(2)_{p-2}$ on $\KQ$ up to conjugacy is a graded separable action of $\PSU(2)_{p-2}$ on $\KQ$. 
\end{thm}

\begin{thm}
    Fix a quiver $Q$. If there exists a separable based action of $\PSU(2)_{p-2}$ on $\KQ$ up to conjugacy then there exists a partition of the vertices into $n$ sets $S_1,...,S_n$ of size $\frac{p-1}{2}$ such that for each $S_k$ we pick a bijection $\psi_k:S_k\rightarrow \Irr(\PSU(2)_{p-2}$ that maps each vertex $v_{l_k}$ to an isomorphism class of simple object $X \in \Irr(C)$ and the subquiver $Q_{ij}$ of all paths from $S_i$ to $S_j$ can be expressed as $|(Q_{ij})_{l_i,m_{j}}|=N_{Z_{ij},\psi_i(v_{l_i})}^{\psi_j(v_{m_j})}$ for some isomorphism class of objects $Z_{ij}\in C$.
\end{thm}

\section{\texorpdfstring{\centering Preliminaries}{Preliminaries}}

In this section, we will provide a brief overview of the background material needed in this paper and establish some notation. For more in depth coverage of this material, refer to  \protect\cite{EGNO16} and \protect\cite{ASS06}. 

\subsection{\texorpdfstring{\centering Quivers and Path Algebras}{Quivers and Path Algebras}}

A quiver $Q=(V,E,s,t)$ consists of a vertex set $V$, edge set $E$, a target map $t:E\rightarrow V$ and a source map $s:E\rightarrow V$. In particular, a quiver is equivalent to a directed multigraph. Given a quiver, we can construct an associative algebra called the path algebra $\KQ$.

\begin{define}[Path Algebra]
    Let $Q$ be a Quiver. Let $\K$ be a field. The Path Algebra $\KQ$ is an $\K$ algebra whose basis is the set of all paths of length $l\geq0$ in $Q$. The product of paths is defined as
    $$e_{i_1}...e_{i_n}*e_{j_1}....e_{j_m}= \delta_{s(e_{i_n})t(e_{j_1})}e_{i_1}...e_{i_n}e_{j_1}....e_{j_m}$$
\end{define}

This gives us an associative algebra and if the quiver is a finite quiver with vertices $v_1,...,v_n$ it follows this is a unital associative algebra where the unit is $1=v_1+...+v_n$. In addition, the vertices form a complete set of  primitive orthogonal idempotents in the path algebra.
Given a quiver, $Q$ we will define the  quiver generated by $Q$. This quiver will be an important concept in constructing actions of $C$ on $\KQ$.

\begin{define} [The quiver generated by $\widetilde{Q}$]
Given a quiver $Q$ we define the quiver generated by $Q$ denoted $\widetilde{Q}$ as the quiver $\widetilde{Q}= \bigoplus_{i=0}^\infty Q^i$ where $Q^i$ is a quiver defined by the paths of length $i$ in $Q$.
\end{define}

\begin{remark}
    $Q^0$ is included here to ensure we account for the vertices. Equivalently, we could define $\widetilde{Q}$ as the quiver that takes a path in $Q$ and makes it an edge in $\widetilde{Q}$ and includes the identity loop at each vertex.
\end{remark}

\begin{example}
Given the following quiver $Q$
    \tikzset{every picture/.style={line width=0.75pt}} 
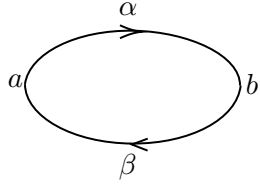
\begin{center}
\begin{tikzpicture}[x=0.75pt,y=0.75pt,yscale=-1,xscale=1]

\draw  [draw opacity=0] (240.16,71.47) .. controls (242.28,55.85) and (266.15,43.92) .. (295.04,44.48) .. controls (322.98,45.03) and (345.78,57.08) .. (348.6,72.08) -- (294.45,74.48) -- cycle ; \draw   (240.16,71.47) .. controls (242.28,55.85) and (266.15,43.92) .. (295.04,44.48) .. controls (322.98,45.03) and (348.78,57.08) .. (348.6,72.08) ;  
\draw  [draw opacity=0] (348.6,72.08) .. controls (348.07,57.21) and (323.4,101.52) .. (294.51,101.43) .. controls (264.57,101.34) and (240.31,87.94) .. (240.16,71.47) -- (294.61,71.43) -- cycle ; \draw   (348.94,72.08) .. controls (347.07,89.21) and (323.4,101.52) .. (294.51,101.43) .. controls (264.57,101.34) and (240.31,87.94) .. (240.16,71.47) ;  
\draw    (299,45) ;
\draw [shift={(299,45)}, rotate = 180] [color={rgb, 255:red, 0; green, 0; blue, 0 }  ][line width=0.75]    (10.93,-3.29) .. controls (6.95,-1.4) and (3.31,-0.3) .. (0,0) .. controls (3.31,0.3) and (6.95,1.4) .. (10.93,3.29)   ;
\draw   (302,105) -- (294.1,101.91) -- (302.05,98.94) ;

\draw (286,105) node [anchor=north west][inner sep=0.75pt]   [align=left] {$\displaystyle \beta $};
\draw (286,29) node [anchor=north west][inner sep=0.75pt]   [align=left] {$\displaystyle \alpha $};

\draw (230,65) node [anchor=north west][inner sep=0.75pt]   [align=left] {$\displaystyle a $};
\draw (350,65) node [anchor=north west][inner sep=0.75pt]   [align=left] {$\displaystyle b $};

\end{tikzpicture}
\end{center}

Then the quiver generated by $Q$, $\widetilde{Q}$ will look like
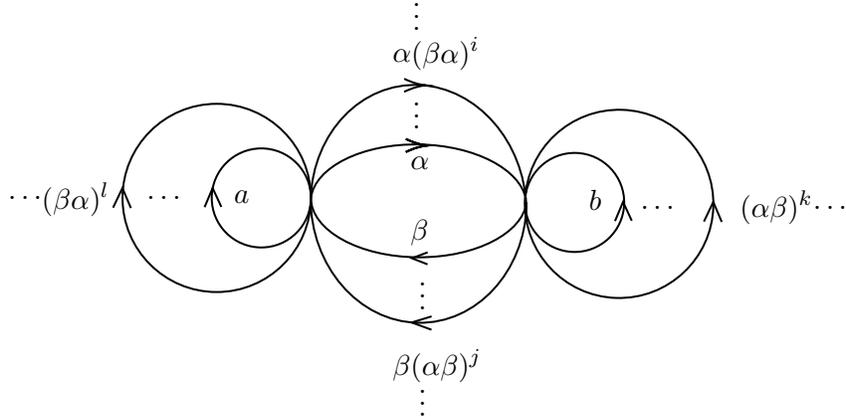
\begin{center}

\tikzset{every picture/.style={line width=0.75pt}} 

\begin{tikzpicture}[x=0.75pt,y=0.75pt,yscale=-1,xscale=1]

\draw  [draw opacity=0] (260.16,95.47) .. controls (262.28,79.85) and (286.15,67.92) .. (315.04,68.48) .. controls (342.98,69.03) and (365.78,81.08) .. (368.6,96.08) -- (314.45,98.48) -- cycle ; \draw   (260.16,95.47) .. controls (262.28,79.85) and (286.15,67.92) .. (315.04,68.48) .. controls (342.98,69.03) and (365.78,81.08) .. (368.6,96.08) ;  
\draw  [draw opacity=0] (368.94,97.55) .. controls (367.07,113.21) and (343.4,125.52) .. (314.51,125.43) .. controls (284.57,125.34) and (260.31,111.94) .. (260.16,95.47) -- (314.61,95.43) -- cycle ; \draw   (368.94,97.55) .. controls (367.07,113.21) and (343.4,125.52) .. (314.51,125.43) .. controls (284.57,125.34) and (260.31,111.94) .. (260.16,95.47) ;  
\draw    (319,69) ;
\draw [shift={(319,69)}, rotate = 180] [color={rgb, 255:red, 0; green, 0; blue, 0 }  ][line width=0.75]    (10.93,-3.29) .. controls (6.95,-1.4) and (3.31,-0.3) .. (0,0) .. controls (3.31,0.3) and (6.95,1.4) .. (10.93,3.29)   ;
\draw    (319,69) ;
\draw [shift={(319,69)}, rotate = 180] [color={rgb, 255:red, 0; green, 0; blue, 0 }  ][line width=0.75]    (10.93,-3.29) .. controls (6.95,-1.4) and (3.31,-0.3) .. (0,0) .. controls (3.31,0.3) and (6.95,1.4) .. (10.93,3.29)   ;
\draw    (319,69) ;
\draw [shift={(319,69)}, rotate = 180] [color={rgb, 255:red, 0; green, 0; blue, 0 }  ][line width=0.75]    (10.93,-3.29) .. controls (6.95,-1.4) and (3.31,-0.3) .. (0,0) .. controls (3.31,0.3) and (6.95,1.4) .. (10.93,3.29)   ;
\draw   (319,129) -- (311.1,125.91) -- (319.05,122.94) ;
\draw  [draw opacity=0] (260.05,95.7) .. controls (261.36,62.42) and (285.7,36.97) .. (315.19,38.5) .. controls (344.22,40.02) and (367.63,67.14) .. (368.82,99.83) -- (314.43,100.48) -- cycle ; \draw   (260.05,95.7) .. controls (261.36,62.42) and (285.7,36.97) .. (315.19,38.5) .. controls (344.22,40.02) and (367.63,67.14) .. (368.82,99.83) ;  
\draw  [draw opacity=0] (368.92,99.42) .. controls (367.88,132.74) and (343.74,159.01) .. (314.24,158.48) .. controls (284.23,157.93) and (260,129.85) .. (260.05,95.7) -- (314.5,96.48) -- cycle ; \draw   (368.92,99.42) .. controls (367.88,132.74) and (343.74,159.01) .. (314.24,158.48) .. controls (284.23,157.93) and (260,129.85) .. (260.05,95.7) ;  
\draw    (318.15,38.62) ;
\draw [shift={(318.15,38.62)}, rotate = 180] [color={rgb, 255:red, 0; green, 0; blue, 0 }  ][line width=0.75]    (10.93,-3.29) .. controls (6.95,-1.4) and (3.31,-0.3) .. (0,0) .. controls (3.31,0.3) and (6.95,1.4) .. (10.93,3.29)   ;
\draw   (321,162) -- (311,158.34) -- (321.04,155.25) ;
\draw   (210.16,95.47) .. controls (210.16,81.66) and (221.35,70.47) .. (235.16,70.47) .. controls (248.96,70.47) and (260.16,81.66) .. (260.16,95.47) .. controls (260.16,109.28) and (248.96,120.47) .. (235.16,120.47) .. controls (221.35,120.47) and (210.16,109.28) .. (210.16,95.47) -- cycle ;
\draw   (165.16,95.47) .. controls (165.16,69.24) and (186.42,47.97) .. (212.66,47.97) .. controls (238.89,47.97) and (260.16,69.24) .. (260.16,95.47) .. controls (260.16,121.7) and (238.89,142.97) .. (212.66,142.97) .. controls (186.42,142.97) and (165.16,121.7) .. (165.16,95.47) -- cycle ;
\draw   (418.16,98.55) .. controls (417.71,112.35) and (406.16,123.18) .. (392.36,122.73) .. controls (378.56,122.29) and (367.74,110.74) .. (368.18,96.94) .. controls (368.63,83.14) and (380.18,72.31) .. (393.98,72.76) .. controls (407.78,73.2) and (418.6,84.75) .. (418.16,98.55) -- cycle ;
\draw   (463.13,100) .. controls (462.28,126.22) and (440.34,146.79) .. (414.12,145.95) .. controls (387.9,145.1) and (367.33,123.16) .. (368.18,96.94) .. controls (369.03,70.72) and (390.97,50.15) .. (417.19,51) .. controls (443.41,51.84) and (463.98,73.78) .. (463.13,100) -- cycle ;
\draw   (206.54,100.14) -- (210.61,90.47) -- (215,100) ;
\draw   (160.39,100.4) -- (165.11,90.5) -- (169.39,100.6) ;
\draw   (413.54,107.14) -- (417.61,97.47) -- (422,107) ;
\draw   (458.39,107.4) -- (463.11,97.5) -- (467.39,107.6) ;

\draw (309,105) node [anchor=north west][inner sep=0.75pt]   [align=left] {$\displaystyle \beta $};
\draw (309,73) node [anchor=north west][inner sep=0.75pt]   [align=left] {$\displaystyle \alpha $};
\draw (300,170.13) node [anchor=north west][inner sep=0.75pt]  [rotate=-0.32] [align=left] {$\displaystyle \beta ( \alpha \beta )^{j}$};
\draw (300,12.38) node [anchor=north west][inner sep=0.75pt]  [rotate=-0.32] [align=left] {$\displaystyle \alpha ( \beta \alpha )^{i}$};
\draw (399,90) node [anchor=north west][inner sep=0.75pt]   [align=left] {$\displaystyle b$};
\draw (475,90) node [anchor=north west][inner sep=0.75pt]   [align=left] {$\displaystyle ( \alpha \beta )^{k}$};
\draw (123,86) node [anchor=north west][inner sep=0.75pt]   [align=left] {$\displaystyle ( \beta \alpha )^{l}$};
\draw (220,90) node [anchor=north west][inner sep=0.75pt]   [align=left] {$\displaystyle a$};
\draw (176,93) node [anchor=north west][inner sep=0.75pt]   [align=left] {$\displaystyle \dotsc $};
\draw (106,93) node [anchor=north west][inner sep=0.75pt]   [align=left] {$\displaystyle \dotsc $};
\draw (425,98) node [anchor=north west][inner sep=0.75pt]   [align=left] {$\displaystyle \dotsc $};
\draw (512,98) node [anchor=north west][inner sep=0.75pt]   [align=left] {$\displaystyle \dotsc $};
\draw (313,128) node [anchor=north west][inner sep=0.75pt]   [align=left] {$\displaystyle \vdots $};

\draw (313,182) node [anchor=north west][inner sep=0.75pt]   [align=left] {$\displaystyle \vdots $};
\draw (310,37) node [anchor=north west][inner sep=0.75pt]   [align=left] {$\displaystyle \vdots $};
\draw (310,-13) node [anchor=north west][inner sep=0.75pt]   [align=left] {$\displaystyle \vdots $};

\end{tikzpicture}

\end{center}

Quivers occur very naturally when working with semisimple linear categories. If we have an endofunctor $G:C\rightarrow C$ and $C$ is finitely semisimple (finitely many isomorphism classes of simple objects) then we can create a quiver $Q$ where the vertices are the simple objects $\{X_1,...X_n\}$ and the edges correspond to the image $G(X_1)$ where if $X_2$ is in the image of $G(X_1)$ then we have an edge going from $X_1$ to $X_2$ in our quiver. In particular, by \protect\cite{KV94}, there is a bijection between linear endofunctors of a semisimple category with $n$ isomorphism classes of simple objects and quivers defined on $n$ vertices. Thus, the question of fusion categories acting on path algebras is a natural question since there is already have a relation between quivers and endofunctors. 
\end{example}

\subsection{\texorpdfstring{\centering Tensor Categories}{Tensor Categories}}

This subsection covers the necessary tensor category theory we will need for our results. Many definitions are taken from \protect\cite{EGNO16}. For the rest of the paper $X,Y,Z$ are objects in our category $C$. An object is simple if $X$ has only trivial subobjects. When $\K$ is an algebraically closed field $X$ is simple if $\Hom(X,X)\cong \K$. A category $C$ is semisimple if all elements are isomorphic to a direct sum of simple objects.

A monoidal category is a quintuple $(C,\otimes,a,\1,i)$ where $C$ is a category $\otimes:C \times C \rightarrow C$ is a bifunctor, there is a unit object $\1$ and natural isomorphisms $a:(X\otimes Y)\otimes Z \rightarrow X\otimes (Y\otimes Z)$ and $i:\1 \otimes \1 \rightarrow \1$
satisfying the pentagon and triangle diagrams. Some examples of monoidal categories are $\VecC,$ the category of vector spaces and $\Set$ the category of sets.
 An object $X^*$ in C is said to be a left dual of $X$ if there exist morphisms $\ev_X : X^* \otimes X \rightarrow \1$ and $\coev_X : \1 \rightarrow X \otimes X^*$, called the evaluation and coevaluation, such that the compositions satisfy the usual zigzag relations. Right duals are defined analogously. Let $C^\circ$ be the full monoidal subcategory of dualizable objects of $C$.

 Let $(C, \otimes, \1, a, i)$ and $(C', \otimes', \1', a', i')$ be two monoidal categories. A monoidal functor from $C$ to $C'$ is a pair $(F,J_{X,Y})$ where $F$ is a functor from $C$ to $C'$ and $J_{X,Y}$ is a natural isomorphism between  $F(X)\otimes'F(Y)$ and $F(X \otimes Y)$. It follows that $J_{X,Y}$ constructs an isomorphism $j:F(\1)\rightarrow \1'$. If $J_{-,-}$ and $j$ are just morphisms as opposed to isomorphisms, then we call $(F,J_{-,-})$ a lax monoidal functor.

Many of the linear categories we will work with are finitely semisimple, of which there is a nice classification of. By the Yoneda embedding all finitely semisimple linear categories with $n$ simple objects are equivalent as a category to $\mathbf{\VecC}^{\oplus n}$ i.e $C \cong (\VecC,\VecC,...,\VecC)$. This implies that if $C,D$ are finitely semisimple linear categories with $n$ and $m$ isomorphism classes of simple objects, then $\Fun(C,D)$ is equivalent to the category of $n\times m$ matrices with $\VecC$ in each entry. Natural transformations between functors $F:C\rightarrow D$ can be represented as a $n\times m$ matrix of linear transformations.

  A category $C$ is a fusion category if $C$ is a finite $\K$-linear abelian rigid semisimple monoidal category such that $\Hom(\1,\1)\cong \K$. Examples of fusion categories are $\VecC(G)$, finite dimensional $G$ graded vector spaces and $\Rep(G)$, finite dimensional representations of $G$. A more abstract example of a fusion category is $\Fib$ the Fibonacci Category. This is a category with two simple objects $\1$ and $\tau$ satisfying the fusion rule that $ \tau\otimes \tau\cong \1 \oplus \tau$. \protect\cite{BD12} provides an in depth description  of the fusion categorical data in $\Fib$.

\subsection{\texorpdfstring{\centering Algebra Objects and Module Categories}{Algebra Objects and Module Categories}}

One view of tensor categories is the categorification of a ring, from that perspective we can then try to understand how modules over that ring can be categorified. A left module category $M$ over a monoidal category $C$ is a category $M$ equipped with a bifunctor $\triangleright:C \times M \rightarrow M$ and a natural isomorphism $m_{X,Y,M}:(X\otimes Y)\triangleright M \rightarrow X\triangleright (Y\triangleright M)$ satisfying the pentagon and triangle axioms. Similar to ring theory, the canonical $C$-Module category is $C$ acting on itself via left action. From \protect\cite[Chapter 7] {EGNO16} there is a bijection between structures of a $C$-module categories on M and monoidal functors $F : C \rightarrow  \End(M)$.

This a categorification of how we think of modules over a ring. A ring $R$ acting on a module $M$ is equivalent to the data of a ring homomorphism from $R$ into $\End(M)$ the endomorphism ring of $M$. Now we will state a quintessential theorem from \protect\cite{Wat} and \protect\cite{Eil60}. 
    Let $R$ and $S$ be two rings, then 
\[\Bim(R,S) \cong \Fun_{coc}(\Mod_R,\Mod_S).\]
Where $\Fun_{coc}$ is the category of colimit preserving additive functors.

This is an essential theorem for studying actions of fusion categories on algebras. If we have a $C$ module category structure on $\Mod(A)$ then we have a monoidal functor $F:C\rightarrow \End(\Mod(A))$. By the Eilenberg-Watts theorem, it follows that there is a monoidal functor $F:C\rightarrow \Bim(A)$. In particular, since monoidal functor preserve duals, it follows that $F:C\rightarrow \Bim(A)^\circ$. Note that all functors we will study in $\End(\Mod(A))$ will preserve colimits and are additive.

Let $M$ and $N$ be $C$-Module categories. A $C$-Module functor between $M$ and $N$ is a functor $G$ and a natural isomorphism $\eta_{X,M}:G(X\triangleright M) \rightarrow X\triangleright G(M)$ satisfying some coherences. Given a tensor category $C$ and a module category, $M$ the dual category $\End_C(M)$ is the category of $C$-Module endofunctors on $M$. The dual category $\End_C(M)$ is also a tensor category, in particular if $C$ is a fusion category
  and $M$ is a semisimple $C$ module category by \protect\cite{EGNO16} then $\End_C(M)$ is also a multifusion category. If in addition $M$ is indecomposable, then $\End_C(M)$ is a fusion category.
    Let $C$ be a linear monoidal category. An internal algebra object $A$ to the tensor category $C$ is a triple $(A, m, u)$, where A is an object in C, and $m : A \otimes A \rightarrow A$ and $u : \1 \rightarrow A$ are morphisms such that they satisfy the standard algebra coherences. An associative algebra is an algebra object in the category $\VecC$. A right module over an algebra $(A, m, u)$ in C is a pair $(M, p)$, where M is an object in C and $p : M \otimes A \rightarrow M$ is a morphism satisfying the usual module coherences. There is a natural left $C$-Module structure on $\Mod(A)$ for some internal algebra $A$ using the module associator and our $A$ action. We can understand all semisimple module categories in a fusion category using the machinery of algebra objects and their category of right modules by the following two results. 
    Let $C$ be a fusion category and $M$ be a semisimple left module category then $M$ then $M\cong \Mod(A)$ where $A$ is an algebra object in $C$ \protect\cite{Ost03} and if $M \cong \Mod(A)$ is a left $C$-Module category, then $\End_C(M) \cong \Bim(A)^{op}$ the opposite tensor product category \protect\cite{EGNO16}.

In practice, finding all the semisimple module categories of a fusion category is not a trivial task. We define a special type of fusion category with nice semisimple $C$ module category structures.

\begin{define}
    A fusion category $C$ is algebra complete if every semisimple module category of $C$ is equivalent to $C^{\oplus n}$ as $C$ module categories.
\end{define}
\begin{remark}
    In other places of the literature these are also referred to as torsion free fusion categories, see \protect\cite{ADC19} for an in depth discussion of them.
\end{remark}
    $\Fib$ the Fibonacci category is a member of a family algebra complete fusion category $\PSU(2)_{p-2}$. This infinite family of algebra complete fusion categories will show up later in the paper
 
\section{\texorpdfstring{\centering The Two Category of Based Algebra Actions}{The Two Category of Based Algebra Actions}}

In this section, we define a based action of a fusion category on an associative algebra. We also discuss a two category of based actions and construct notions of equivalences of these actions. Then we motivate our interest in based actions by connecting them to Hopf algebra actions.  

A group action of $G$ on $A$ is a group homomorphism $\phi:G\rightarrow \Aut(A)$. Each automorphism $\phi_g$ is generalized by a bimodule $A_{\phi_g}$ where the left action is multiplication in $A$ and the right action is twisted by the automorphism $\phi_g$ such that $A_{\phi_g}\otimes_A A_{\phi_h}=A_{\phi_{gh}}$. Fusion category actions on associative algebras generalize this notion.
An action of a fusion category on an associative algebra is a $C$ module structure on $\Mod(A)$. Using the Eilenberg-Watts Theorem, this produces our definition of an action of a fusion category $C$ on $A$.
\begin{define}
    An action of a fusion category $C$ on an associative algebra $A$ is a monoidal functor $F:C\rightarrow \Bim(A)$.
\end{define}

When we categorify the idea of a finite group, we get $\VecC(G)$.
An action of $\VecC(G)$ on $A$ is a monoidal functor $F:\VecC(G)\rightarrow \Bim(A)$. A natural question is how do actions of $G$ on $A$ relate to $\VecC(G)$ actions on $A$? We will show that based $\VecC(G)$ actions fully generalize group actions. More generally, we will show that for a semisimple Hopf algebra $H$ that based $\coRep(H)$ actions on $A$ generalize $H$ actions on $A$. This motivates our interest in defining based actions.

\begin{define}
    A based action of $C$ on $A$ is defined by a triple $(A,F,V_{-})$ where $(F,J_{-,-}):C\rightarrow \Bim(A)$ is a monoidal functor and $V_X\subset F(X)$ is a choice of finite dimensional subspace such that $\1_A\in V_1$, and morphisms preserve these subspaces, that is if $f:F(X)\rightarrow F(Y)$ then $f(V_X)\subset V_Y$.
\end{define}

\begin{lem} \label{lem:based action V_X}
    The subspace $V_X$ spans a left/right projective basis for $F(X)$.
\end{lem}

\begin{proof}
    We will prove this result for right duals and the result for left duals is analogous. Since $F$ is a monoidal functor from, $C \rightarrow \Bim(A)$ it follows that $F(X)$ is dualizable for each $X$. In particular, we define the evaluation and coevaluation as follows, 

    \[\ev_{F(X)}:=F(X)\otimes F(X^*)\xrightarrow{J_{X,X^*}} F(X\otimes X^*)\xrightarrow{\ev_X}F(\1)\]
\[\coev_{F(X)}:=F(\1)\xrightarrow[]{\coev_X} F(X\otimes X^*)\xrightarrow[]{J_{X,X^*}^{-1}} F(X)\otimes F(X^*).\]

Notice that since $1_A\in V_1$ then $\coev_{F(X)}(1_A)\in V_{X}\otimes_A V_{X^*}$. In particular, this implies that $\coev_{F(X)}(1_A)=\sum_i v_i\otimes_A v_i^*$ where $v_i \in V_X$ and $v_i^*\in V_{X^*}$. Since ev and coev satisfy the zigzag relations, it follows that $\{v_i\}$ form a projective basis of $F(X)$.
\end{proof}

\begin{cor}\label{cor: coev lives in V_1}
    Let $z\in V_X$, then $z=\sum_ix_i\alpha_i$ where $\alpha_i\in V_{\1}$.
\end{cor}
\begin{proof}
    If $z\in V_X$ then $z\in X$ which implies that $z=\sum_i x_ix_i^*(z)$. But since the evaluation morphism sends $V_X\otimes_A V_X^*$ to $V_1$ then it follows that $x_i^*(z)\in V_1$ for each $i$.
\end{proof}

There is a natural notion of a two category of algebras. The objects are algebras, the $1$ morphisms are bimodules and then the two morphisms are bimodule homomorphisms. We now define the two category of based actions of a fusion category on an algebra, which we will denote $C-\BasedAlg$. We model our category off the groundwork laid down in \protect\cite{CJHP24} with their two category of $C^*$ algebra actions. 
\begin{define}
    Fix a fusion category $C$. $C-\BasedAlg$ is the two category defined as follows,
    \begin{enumerate}
        \item Objects are associative algebras $A$ equipped with based actions of $C$.

        \item One morphisms are based $B-A$ bimodules, that is triples $(Q,V_Q,\phi_{-})$. If we have based actions $F:C\rightarrow \Bim(A)$ and $G:C\rightarrow \Bim(B)$ then $Q$ is an object in $\Bim(B-A)$ and $\phi_X$ is a family of half braidings which satisfy the relation $\phi_X:Q \otimes_A F(X) \xrightarrow{\sim} G(X) \otimes_B Q$ such that is restricts to $\phi_X:V_Q \otimes_A V^F_{X} \xrightarrow{\sim} V_X^G \otimes_B V_Q$ and the diagram below commutes, (suppressing associators)

        \[\adjustbox{scale=0.90,center}{ \begin{tikzcd}
	& {Q \otimes_{A}F(X)\otimes_{A} F(Y)} \\
	{G(X) \otimes_B Q \otimes_{A}G(Y)} && {Q\otimes_A F(X\otimes Y) } \\
	\\
	{G(X) \otimes_{B}G(Y)\otimes_B Q} && {G(X\otimes Y) \otimes_{B} Q}
	\arrow["{\phi_X \otimes_A\id_{F(Y)}}", from=1-2, to=2-1]
	\arrow["{\id_Q \otimes J^F_{X,Y}}"', from=1-2, to=2-3]
	\arrow["{\id_{G(X)}\otimes \phi_Y}", from=2-1, to=4-1]
	\arrow["{\phi_{X,Y}}"', from=2-3, to=4-3]
	\arrow["{ J^G_{X,Y}\otimes_B Q}", from=4-1, to=4-3]
\end{tikzcd}}.\]
        
        \item Two morphisms are bimodule intertwiners compatible with the action of $C$ that preserve the $V_Q$ spaces. That is $f:(Q,V_Q,\phi_{-})\rightarrow (Q', V_{Q'},\phi'_{-})$ such that $( id_{G(X)} \otimes_B f)\circ \phi_X=\phi_X'\circ (f\otimes_A \id_{F(X)})$ for all $X$ in $C$ and $f(V_Q)\subset V_{Q'}$.
        
        \item Let $(A,F,V^F),(B,G,V^G)$ and $(C,H,V^H)$ be objects in $C-\BasedAlg$, $(Q,V_Q,\phi)\in \Bim(B-A)$ and $(Q',V_{Q'},\phi') \in \Bim(C-B)$ be one morphisms then the composite morphism is defined by 

        \[(Q'\otimes_B Q,(\phi'_X \otimes_B \id_Q)\circ (\id_{Q'}\otimes_B \phi_X)).\]

        \item Given $(Q,V_Q),(Q',V_{Q'})$ in $\Bim(B-A)$ and $(P,V_P),(P',V_{P'}) \in \Bim(C-B)$ with the corresponding intertwiners $f:Q\rightarrow Q'$ and $g:P\rightarrow P'$, horizontal composition of $f$ and $g$ is $g\otimes_B f$.
    \end{enumerate}

\end{define}

We have defined this two category of actions. We are interested in when these actions are equivalent and what different levels of equivalence of based actions exist. In the literature of fusion categories acting on algebras, there are three different levels of equivalence.

\begin{define}
     Two objects $(A,F,V)$ and $(A,F',V')$ are equivalent if there exists a monoidal natural isomorphism $\gamma$ from $F$ to $F'$ that preserves the base spaces.
\end{define}

\begin{define}
  Two objects $(A,F,V)$ and $(A,F',V')$ are conjugate, if there is an invertible objects $u\in A$ such that conjugation by $u$, $\phi_u:A\rightarrow A'$ such that $F^u$ is monoidally naturally isomorphic to $F'$ such that it  preserves the base spaces. If $X \in \Bim(A)$ then $X^u \in \Bim(A)$ where the action is $a_1\cdot x\cdot a_2$ is define as $\phi_u(a_1)\cdot x\cdot \phi_u(a_2)$. 
\end{define}

\begin{define}
    Two objects $(A,F,V)$ and $(A',F',V')$ are Morita equivalent if there exists an invertible one morphism $(Q,V_Q,\phi_{-})$ in $\Bim(A',A)$ such that $(Q,V_Q)\otimes_A (\overline{Q},\overline{V_Q})\cong (A',V_A')$ and $(\overline{Q},\overline{V_Q})\otimes_{A'}(Q,V_Q)\cong (A,V_A)$ .
\end{define}

\begin{remark}
    These are all stronger variants of the equivalent definitions for fusion category actions on algebras. For example, an equivalent based action of $C$ on $A$ is an equivalent action of $C$ on $A$.
\end{remark}

Let $(A,F,V_{-})$ be a based action. Notice that we have an assignment from $X\rightarrow V_{X}$. We will show that this produces a lax monoidal functor from $C\rightarrow \VecC.$

\begin{lem} \label{lem:lax monoidal functor}
    Let $(A,F,V_{-})$ be a based action of $C$ on $A$, then there is a lax monoidal functor $(F',J'_{-,-}):C\rightarrow \VecC$.
\end{lem}
\begin{proof}
Since $F:C\rightarrow \Bim(A)$ is a monoidal functor then, $J_{X,Y}(F(X)\otimes_A F(Y))\cong F(X\otimes Y)$ it follows that by the definition of relative tensor product we induce a $A$ balanced morphism from $\tilde{J}_{X,Y}:F(X)\otimes F(Y)\rightarrow F(X\otimes Y)$ by the universal property of relative tensor products. In particular $V_X\otimes V_Y\xrightarrow{\tilde{J}_{X,Y}} V_{X\otimes Y}$. This induces a mapping $F':C\rightarrow \VecC$ that sends $X\rightarrow V_X$ and a morphism $J'_{X,Y}=\tilde{J}|_{V_X\otimes V_Y}$. 

We define our lax monoidal functor $F':C\rightarrow \VecC$ that sends $X\rightarrow V_X$ and $J'_{X\otimes Y}:F(X)\otimes F(Y)\rightarrow F(X\otimes Y)$. 

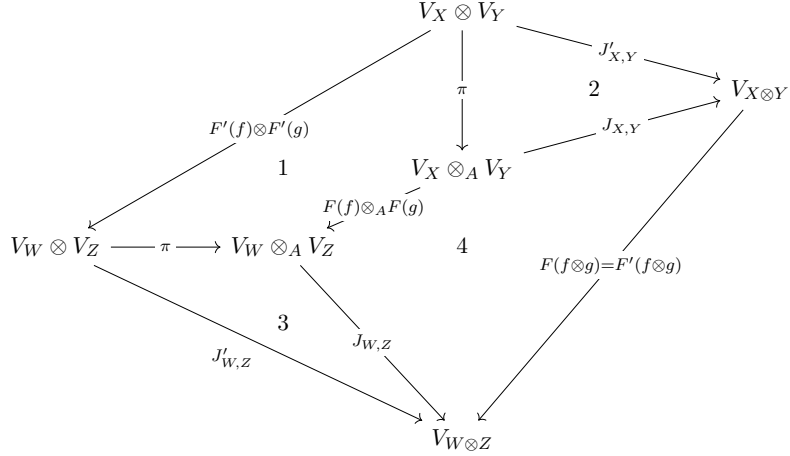
\[\adjustbox{scale=.8,center}{\begin{tikzcd}
	&&& {V_X\otimes V_Y} \\
	&&&& 2 && {V_{X\otimes Y}} \\
	&& 1 & {V_X\otimes_A V_Y} \\
	{V_{W}\otimes V_{Z}} && {V_{W}\otimes_A V_{Z}} & 4 \\
	&& 3 \\
	\\
	&&& {V_{W\otimes Z}}
	\arrow["{J'_{X,Y}}"{description}, from=1-4, to=2-7]
	\arrow["\pi"{description}, from=1-4, to=3-4]
	\arrow["{F'(f)\otimes F'(g)}"{description}, from=1-4, to=4-1]
	\arrow["{F(f\otimes g)=F'(f\otimes g)}"{description}, from=2-7, to=7-4]
	\arrow["{{J}_{X,Y}}"{description}, from=3-4, to=2-7]
	\arrow["{F(f)\otimes_A F(g)}"{description}, from=3-4, to=4-3]
	\arrow["\pi"{description}, from=4-1, to=4-3]
	\arrow["{J'_{W,Z}}"', from=4-1, to=7-4]
	\arrow["{J}_{W,Z}"{description}, from=4-3, to=7-4]
\end{tikzcd}}\]

$1,2,3$ commute by definition of relative tensor product, and $4$ commutes by the naturality of $J$. Thus, $J'$ is a natural transformation.
By another commutative diagram argument using the definition of $A$ balanced tensor product, similar to the one above, $J'$ will satisfy the necessary coherences to make $(F',J')$ a lax monoidal functor.   
\end{proof}
\begin{cor} \label{cor: alg object in cop}
    Let $(A,F,V)$ be a based action of a fusion category $C$ on $A$. Then $F'$ has the structure of an internal algebra object in $C^{op}$.
\end{cor}
\begin{proof}
   By \protect\cite[Proposition 3.3]{JP17},
 this is the same data as an algebra object in $C^{op}$.
\end{proof}

\begin{cor}
    $F'(\1)\cong V_\1$ is a finite dimensional unital algebra in $\VecC$. 
\end{cor}
\begin{proof}
    Follows since $\1$ is an algebra in $C$.
\end{proof}

Our initial motivation for defining based actions was to generalize Hopf algebra actions to fusion categories. We shall see later that in the case of a Hopf algebra action on $A$, the based actions they generate $V_\1$ subspace is always a semisimple algebra.

\begin{define}
    A separable based action is a based action $(A,F,V)$ such that $V_\1$ is a semisimple algebra in $\VecC$.
\end{define}

\begin{define}
 Based actions of $C$ on $A$ over a functor $G$ are based actions such that the corresponding lax monoidal functor $F'\cong G$. 
\end{define}

Using this definition, we will now have the data necessary to construct $H$ actions on $A$ out of based actions of $\coRep(H)$ on $A$ for a semisimple Hopf algebra $H$. We will first show how a $H$ module algebra produces a monoidal functor $F:\coRep(H)\rightarrow \Bim(A)$.

\begin{lem}\label{lem:coRep becomes bim}
 Let $A$ be a $\Rep(H)$ module algebra and let $V$ be a corepresentation of $H$, then $A\otimes V$ is an $A-A$ bimodule.
\end{lem}

\begin{proof}
    Let $V$ be a corepresentation of $H$. Then we can define an $A$ bimodule by taking $A \otimes V$. The left action of $A$ on $A \otimes V$ is just left multiplication. The right action of $A$ on $A\otimes V$ is done by applying the comodule action on $V$ giving us $\sum_i(a\otimes v_i\otimes h_i)\triangleleft a_0:=\sum_iah_i(a_0)\otimes v_i$ where we apply the coaction of $H$ on $V$ and then apply $h_i$ to $a_0$ on the left. The left and right actions give module actions and are compatible, thus $A\otimes V$ is a bimodule.
\end{proof}

\begin{lem}\label{lem:hopf action anives based action}
    Let $A$ be a $\Rep(H)$ module algebra, then the action produces a monoidal functor $F:\coRep(H)\rightarrow \Bim(A)$ defined by $F(X)=A\otimes X$.
\end{lem}
\begin{proof}
    First, let's show that $F$ is a functor. Notice that $F(X)$ is a $A-A$ bimodule by the Lemma \protect\ref{lem:coRep becomes bim}. Let $f:X\rightarrow Y$ be a comodule homomorphism, then notice that $F(f)=\id_A\otimes f:A\otimes X\rightarrow A\otimes Y$. This is clearly a left $A$ module homomorphism, thus, we just need to check that it's a right $A$ module homomorphism. This follows since $f$ is a comodule morphism, implying $f$ commutes with the comodule action on $X$ and $Y$.  $F$ respects composition and sends the identity to the identity, thus $F$ is a functor.

    Now we will show that $F$ is monoidal. We need to construct an $\epsilon_{X,Y}:F(X)\otimes_A F(Y) \rightarrow F(X\otimes Y)$. In particular, $\epsilon_{X,Y}:(A\otimes X)\otimes_A(A\otimes Y)\rightarrow A\otimes (X\otimes Y)$. Let, 
    
    \[\epsilon_{X,Y}:=(\id_A \otimes r_X \otimes \id_Y) \circ (\id_A \otimes \alpha_{X,A,Y})\circ \alpha_{A,X,A\otimes Y}.\]

    This is a natural isomorphism because of the naturality of the associator and the unitor. We can see that $(F,\epsilon_{-,-})$ satisfies all the monoidal functor criterion by computation, thus $(F,\epsilon_{-,-})$ is a monoidal functor.
    \end{proof}

We see that the above lemma implies that $\coRep(H)$ actions generalize Hopf algebra actions. A specific example is when $H=\K G$ the group algebra. Then, this theorem translates to $\VecC(G)$ actions on $A$ generalize group actions of $G$ on $A$. Building on this relation, we would like to know when actions of $\coRep(H)$ on $A$ are $H$ module algebras. Define $F_H$ to be the canonical forgetful functor $F:\coRep(H)\rightarrow \VecC$. Notice, by Tannaka-Krein duality, this equips $\coEnd(F_H)$ with the structure of the Hopf algebra $H$. Notice that if $F_H$ is the structure of $F'$ then $V_{\1}$ is a one dimensional vector space corresponding to $1_A$. 
\begin{prop}
    Suppose we have a based action $\coRep(H)$ on $A$ over $F_H$. Then $V_X$ is the span of a projective basis of $F(X)$ for all isomorphism classes of object $X$ in $\coRep(H)$
\end{prop}
\begin{proof}
By Corollary \protect\ref{cor: coev lives in V_1} it follows that if $z\in V_X$ then $z=\sum_i x_i\alpha_i$ where $\alpha_i$ in $V_{\1}$ but since $V_{\1}$ is the one dimensional vector space corresponding to $1$ then $\alpha_i=1$. Thus, $\{x_i\}$ is a basis for $V_X$ and a projective basis for $F(X)$.
\end{proof}

Now we state and prove some lemmas connecting based actions of $\coRep(H)$ on $A$ over $F_H$ and $H$ actions on $A$.

\begin{lem}
    Given an action of $H$ on $A$ then it produces a  based action of $\coRep(H)$ on $A$ over $F_H$.
\end{lem}
\begin{proof}
    Let $A$ be an $H$ module algebra.  Then we induce a monoidal functor $F:\coRep(H)\rightarrow \Bim(A)$. This produces a based action where the base spaces are $1_A\otimes V$. We can construct a based action of $\coRep(H)$ on $\Bim(A)$ over $F_H$ on $A$ by mirroring the action of $H$ on $A$. As a comodule, notice that each isomorphism class of simple $V$ produces a half braiding with $a$ action on the right via comultiplication. We can just define the half braiding of $(1_a\otimes x)*a$ to be $h_i(a)\otimes x_i$. Observing this for all simple comodules $V$ constructs our based action of $\coRep(H)$ on $A$ over $F_H$.
\end{proof}

\begin{lem} \label{lem:based action give Hopf action}
Let $H$ be a semisimple Hopf algebra. Then a based action of $\coRep(H)$ on $A$ over $F_H$ produces an action of $H$ on $A$.
\end{lem}

\begin{proof}
  Let $F:\coRep(H)\rightarrow \Bim(A)$ be a based action of $\coRep(H)$ over $F_H$ on $A$. This implies that $F_H\cong F':\coRep(H)\rightarrow \VecC$. By Tannaka-Krein duality, we can equip $\coEnd(F_H)$ with the structure of $H$.  Let $h\in H$ and $a\in A$, by our half braiding $ha=\sum a_ih_i$. Since $H$ is a Hopf algebra, it has a comultiplication map that sends $h$ to $\sum_ih_i\otimes h'_i=h_{(1)}\otimes h_{(2)}$. 
  
  We can define our action of $H$ on $A$ by using comultiplication and the half braiding. In particular $ha=h_{(2)}(a)h_{(1)}$, that is $(h_{(2)},a)$ goes to $h_{(2)}(a)$. Consider, 

  \[hka=h(k_{(2)}(a)k_{(1)})=h_{(2)}(k_{(2)}(a))h_{(1)}k_{(1)}\]

  \[hka=(\sum_jv_j)a=v_{(2j)}(a)v_{(1j)}.\]

  These are equal, and in particular our $H$ action is a Hopf algebra action because our multiplication is a coalgebra homomorphism. In particular, 

  \[ \Delta \circ m=(m\otimes m)\circ (\id_H\otimes \sigma_{H,H} \otimes id_H)\circ(\Delta\otimes \Delta)\]
  where $\sigma$ is the braid isomorphism of vector spaces. We see that $1_H=1_A$ thus we see that $1_Ha=a1_H=1_H(a)1_H$ which implies $1_H(a)=a$. Thus, $A$ is a $H$ module. We shall now show that $A$ has the $H$ module algebra structure. This will follow by how we defined our action, the coassociativity of the comultiplication in $H$. That is,

  \[h(ab)=h_{(2)}(ab)h_{(1)}\]

  \[(ha)b=h_{(2)}(a)h_{(1)}b=h_{(2)}(a)h_{(12)}(b)h_{(11)}.\]

 By our bimodule half braiding these are equal, and this gives an H module algebra structure on $A$ because $(\id_H\otimes \Delta)\circ\Delta=(\Delta\otimes\id_H)\circ\Delta.$ Note that $h1_A=1_A h=h$ and $h1_A=h_{(1)}(1_A)h_{(2)}$. Then using the triangle property of coalgebras yeilds that the action of $h_{(1)}(\1_A)=\epsilon(h_{(1)})1_A$.  Therefore, $A$ is a $H$ module algebra. 
  \end{proof}

Thus, for a semisimple Hopf algebra $H$ actions on a generate based $\coRep(H)$ actions over $F_{H}$ and vice versa. In the following theorem, we establish a bijection between these sets up to equivalence of based actions. Notice that if we have a based action $\coRep(H)$ on $A$ over $F_H$ then the corresponding $G:C\rightarrow \Bim(A)$ is monoidally naturally isomorphic to monoidal functor $F$ defined in Lemma \protect\ref{lem:hopf action anives based action}.

\begin{thm}
    Let $H$ be a semisimple Hopf algebra. Equivalence classes of based actions of $\coRep(H)$ on $A$ over $F_H$ are in bijection with actions of $H$ on $A$.
\end{thm}

\begin{proof}
    We need to construct a bijection between these sets. To show that this is an injection take two actions of $H$ on $A$ that produce the same equivalence class of based action of $\coRep(H)$ on $A$. Then since the half braiding of our bimodules is determined by the action of $H$ on $A$ then for both actions to generate the same based action they must have been the same. To show subjectivity, consider an equivalence class of based actions of $\coRep(H)$ on $A$ then by Lemma \protect\ref{lem:based action give Hopf action} that there is a Hopf algebra action of $H$ on $A$ that arises from this based action. In particular, notice that by Lemma \protect\ref{lem:based action give Hopf action} that we construct a based action of $\coRep(H)$ on $A$. This action is equivalent to the based action we started with. Therefore, the map is a bijection and the result follows.
\end{proof}

\begin{cor}
   Let $G$ be a finite group. Equivalence classes of based actions of $\VecC(G)$  on $A$ over $F_{\K G}$ are in bijection with actions of $G$ on $A$.
\end{cor}
\begin{proof}
    $\VecC(G)$ is equivalent as fusion categories to $\coRep(\K 
 G)$ and group actions are  equivalent to $\K G$ actions. 
\end{proof}

This motivates our interest in separable based actions. We want to study separable based actions of fusion categories on the path algebra $\KQ$. This implies that $V_\1$ is a finite dimensional separable algebra, which implies that $V_\1$ is a semisimple unital subalgebra of $\KQ$.  

\begin{prop}
    The only unital semisimple subalgebras of $\KQ$ are $\K^n$.
\end{prop}
\begin{proof}
Let $B$ be a semisimple unital subalgebra on $\KQ$. Then $B \cong \oplus_{c_i} M_{c_i}(\K)$. Consider one of these subalgebras $M_{c_j}(\K)$, this is the simple subalgebra of $c_j\times c_j$ matrix where the diagonal entries correspond to idempotents in $\KQ$. By the path length grading on $\KQ$ all idempotents must contain a vertex $p_v$ in their sum. The non-diagonal entries in $M_{c_j}(\K)$ cannot contain a vertex, since otherwise $E_{ij}^2\neq 0$. Then it follows that $E_{ij}\hspace{1mm} i\neq j$ are sums of paths of size at least $1$.  But for $i\neq j$ the product of $E_{ij}E_{ji}$ must contain at least one vertex, which is not possible since $E_{ij}$ only has paths of length greater than 0.
\end{proof}

Now it follows that in the case of a separable based action of $C$ on $\KQ$ that $V_{\1}\cong \K^n$ as algebras. The algebra $\K^n$ consists of sums of the orthogonal projections. The different possible $V_{\1}$ spaces correspond to different sets of orthogonal idempotents that sum to $1$. For a collection of such idempotents S, we can define a monoidal category $\BasedBim(\KQ)_{S}$. 

\begin{define}
      Define the monoidal category ${\BasedBim}(\KQ)_{S}$ of separable based bimodules as, 

    \begin{itemize}
        \item The unit is $(\KQ,V_{S})$ where $V_{S}$ a separable unital subalgebra of $\KQ$.
        \item Objects are pairs $(X,V_X)$ where $X$ is a dualizable bimodule and $V_X$ is a finite dimensional subspace such that $V_{S}\otimes_{\KQ} V_X\cong V_X\cong V_X\otimes_{\KQ}V_{S}$.
        \item Morphisms are bimodule homomorphisms $f:X\rightarrow Y$ such that $f(V_X)\subset V_Y$.
        \item The tensor product is defined by $(X,V_X) \otimes_{\KQ} (Y,V_Y)\cong(X\otimes_{\KQ} Y,V_{X}\otimes_{\KQ} V_{Y})$.
    \end{itemize}
\end{define}

Given a ring $R$ such that $R=\oplus I_i$ as a right ideal, then there is a decomposition of $1$ into orthogonal idempotents $1=\sum_i p_i$ such that $I_i=p_iR$. Then, we can decompose these idempotents $p_i$ into their primitive components $e_{ij}$, further decomposing $R=\oplus_{ij} J_{ij}$ as the direct sum of indecomposable right ideals. These primitive idempotents will still be pairwise orthogonal because $e_{lk}=1e_{lk}=(\sum_{ij} e_{ij})e_{lk}$, which implies that $e_{ij}e_{lk}$ is in $J_{ij}$, but these summands are direct, thus $e_{ij}e_{lk}=\delta_{ij,lk}e_{lk}$. Therefore, given a collection of orthogonal idempotents, we can refine that set into a complete set of primitive orthogonal idempotents. We can take an idempotent refinement of $V_{S}$ into $V_I:=\sum_{k,j} e_k V_{S} e_j$ where $e_k,e_j$ are a primitive orthogonal idempotent decomposition of $S$. This leads us to define separable idempotent split based bimodules of the path algebra.
\begin{define}
     Define the monoidal category $\BasedBim(\KQ)_{I}$ of separable idempotent split based bimodules as, 

    \begin{itemize}
        \item The unit is $(\KQ,V_I)$ where $V_{I}$ is the idempotent split based space.
        \item Objects are pairs $(X,V_X)$ where $X$ is a dualizable bimodule and $V_X$ is a finite dimensional subspace such that $V_I\otimes_{\KQ} V_X\cong V_X\cong V_X\otimes_{\KQ}V_I$ .
        \item Morphisms are bimodule homomorphisms $f:X\rightarrow Y$ such that $f(V_X)\subset V_Y$.
        \item The tensor product is defined by $(X,V_X) \otimes_{\KQ} (Y,V_Y)\cong(X\otimes_{\KQ} Y,V_{X}\otimes_{\KQ} V_{Y})$.
    \end{itemize}
\end{define}

There is a special idempotent split basis consisting of the vertex projections $\{p_v\}_{v\in \KQ_0}$, we will call this monoidal category $\BasedBim(\KQ)_{I_v}$. We shall see that all other complete sets of primitive orthogonal idempotents are conjugate to these idempotents, and this category is equivalent to some data related to finitely semisimple linear categories.

We can think of $\BasedBim(\KQ)_I$ as the subcategory corresponding to the monoidal corner, where our refined $V_{I}$ subspace acts as the unit. This refinement doesn't change the bimodules themselves but does change the structure of the $V_X$ subspaces in a canonical way by decomposing them into their most idempotent split form. In particular, if we idempotent split a $V_X$ space over a collection of primitive idempotents $\{e_i\}_i$ into $\oplus_{i,j}e_iV_Xe_j$, then as a vector space $V_X\subset \oplus_{i,j}e_iV_Xe_j$. 

\begin{thm}\label{thm:separable action and idempotent}
    Fix a fusion category $C$. If there is a separable based action $(\KQ,F,V)$ then there is a separable idempotent split based action $(\KQ,F,U)$.
\end{thm}
\begin{proof}
 Suppose that $(\KQ,F,V)$ is a separable based action of $C$ on $\KQ$. Then this implies that there is a monoidal functor $F:C\rightarrow \Bim(\KQ)$ such that $X\rightarrow F(X)$ and $V_X\subset F(X)$. Now since $f(V_X)\subset V_Y$ is a bimodule homomorphism it follows that $f(U_X)\subset U_Y$. Similarly, $J_{X,Y}$ is an isomorphism implying that it will be compatible with the idempotent splitting. This gives rise to an idempotent split action $(\KQ,F,U)$.
\end{proof}

\begin{cor}
    Given an idempotent split separable based action $(\KQ,F,V)$, we can recover a separable based action if $F':C\rightarrow \VecC$ has a subalgebra $F'':C\rightarrow \VecC$ that corresponds to the desired separable action $(\KQ,F,U)$.
\end{cor}

\begin{proof}
Given an idempotent split separable based action, the $V_X$ spaces are larger than a separable based action. By Corollary \ref{cor: alg object in cop} $F'$ has the structure of an algebra object in $C^{op}.$ If there were a separable action that this idempotent split action was generated by it would follow that $F''$ would correspond to a unital subalgebra of $F'$.
\end{proof}

Obtaining a full classification of separable based actions of $C$ on $\KQ$ is a challenging question. Since each separable action corresponds to an idempotent split action we will focus on the classification of idempotent split separable based actions of $C$ on $\KQ$. The path algebra $\KQ$ is a semiperfect ring, implying that all complete sets of primitive orthogonal idempotents are conjugate \cite{TL91}. Also, for an idempotent $e\KQ$ is a finitely generated projective module. Since $\KQ$ is semiperfect, all finitely generated projective \cite{TL91} modules are isomorphic to direct sums of modules generated by primitive orthogonal idempotents, which are in turn isomorphic to vertex modules $p_v\KQ$. Thus, up to conjugacy, all idempotent splittings of a separable based action are equivalent to separable based actions over the separable algebra corresponding to the vertex projections. Later, when we provide a full classification of idempotent split separable based actions, this will be up to conjugacy by an invertible element in the path algebra.

\begin{example}
Consider the following quiver.
   \tikzset{every picture/.style={line width=0.75pt}} 
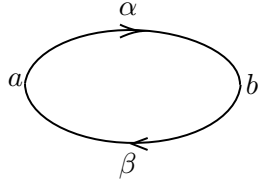
\begin{center}
\begin{tikzpicture}[x=0.75pt,y=0.75pt,yscale=-1,xscale=1]

\draw  [draw opacity=0] (240.16,71.47) .. controls (242.28,55.85) and (266.15,43.92) .. (295.04,44.48) .. controls (322.98,45.03) and (345.78,57.08) .. (348.6,72.08) -- (294.45,74.48) -- cycle ; \draw   (240.16,71.47) .. controls (242.28,55.85) and (266.15,43.92) .. (295.04,44.48) .. controls (322.98,45.03) and (348.78,57.08) .. (348.6,72.08) ;  
\draw  [draw opacity=0] (348.6,72.08) .. controls (348.07,57.21) and (323.4,101.52) .. (294.51,101.43) .. controls (264.57,101.34) and (240.31,87.94) .. (240.16,71.47) -- (294.61,71.43) -- cycle ; \draw   (348.94,72.08) .. controls (347.07,89.21) and (323.4,101.52) .. (294.51,101.43) .. controls (264.57,101.34) and (240.31,87.94) .. (240.16,71.47) ;  
\draw    (299,45) ;
\draw [shift={(299,45)}, rotate = 180] [color={rgb, 255:red, 0; green, 0; blue, 0 }  ][line width=0.75]    (10.93,-3.29) .. controls (6.95,-1.4) and (3.31,-0.3) .. (0,0) .. controls (3.31,0.3) and (6.95,1.4) .. (10.93,3.29)   ;
\draw   (302,105) -- (294.1,101.91) -- (302.05,98.94) ;

\draw (286,105) node [anchor=north west][inner sep=0.75pt]   [align=left] {$\displaystyle \beta $};
\draw (286,29) node [anchor=north west][inner sep=0.75pt]   [align=left] {$\displaystyle \alpha $};

\draw (230,65) node [anchor=north west][inner sep=0.75pt]   [align=left] {$\displaystyle a $};
\draw (350,65) node [anchor=north west][inner sep=0.75pt]   [align=left] {$\displaystyle b $};

\end{tikzpicture}
\end{center}
Separable algebras in this quiver are of the form $\K$ or $\K^2$. If the separable algebra is of the form $\K$ then $\K=1=p_a+p_b$. If $\K^2$ is the separable algebra, then there are multiple algebras. For example, $(p_a+\alpha,p_b-\alpha)$, this is a unital algebra isomorphic to $\K^2$, but this algebra is conjugate to $(p_a,p_b)$ where the invertible element is $1+\alpha$.
\end{example}

\begin{lem} \label{lem: conjugacy of actions}
    If there is a monoidal functor $F:C\rightarrow \BasedBim(\KQ)_{I}$ then it is conjugate to a monoidal functor $F^u:C\rightarrow \BasedBim(\KQ)_{I_v}$, where the idempotent split algebra $I_v$ is the canonical vertex projections
\end{lem}

\begin{proof}
    We have already established that any complete set of primitive orthogonal idempotents are conjugate. Thus, it suffices to show that when we conjugate an idempotent split action to the vertex projections, it's monoidally naturally isomorphism as based actions to an idempotent split based action over the vertex projection.

    Let $F:C\rightarrow \BasedBim(\KQ)_I$ be an idempotent split based action with idempotents $(e_1,...,e_n)$, then there is a basis of $V_X$ such that each $x\in V_X$ is of the form $e_iv_xe_j$. We will conjugate this action by invertible object $u$ that conjugates our primitive idempotents to the vertex projections. We define the conjugate bimodule of $F(X)$ to be $F^u(X)$, this is still a $\KQ$ bimodule. Notice that if we conjugate a monoidal functor $F$ by some invertible element $u$ then we have defined a monoidal functor $F^u:C\rightarrow \Bim(\KQ)$. $J_{X,Y}$ being a natural bimodule isomorphism implies that $J^u_{X,Y}$ is also natural bimodule isomorphism. If there is a morphism $f:F(X)\rightarrow F(Y)$ then this produces a morphism $f^u:F^u(X)\rightarrow F^u(Y)$.
    
    Then notice this conjugation will conjugate our based spaces as well. Given  basis of $V_X$ such that each $x\in V_X$ is of the form $e_iv_xpe_j$ then our new conjugate action must have $ue_iu^{-1}v_xue_ju^{-1}=p_vv_xp_w$ form a basis for our new conjugate $V_X^u$ space for some vertices $v,w\in \KQ_0$. Let $f:F(X)\rightarrow F(Y)$ such that $f(V_X)\subset V_Y$ then it follows that $f(e_iv_xe_j)\xrightarrow{u} f^u(ue_iu^{-1}v_xue_ju^{-1})=f^u(p_uv_xp_w)=p_vf^u(v_x)p_w$ which implies that $f^u$ will also preserve conjugated based spaces. Since $J_{-,-}$ and the unitor are bimodule isomorphisms they will be compatible with the conjugation action by the reasoning shown above. Thus, we have produced an idempotent split based action over the vertex projections. Therefore, it follows that any idempotent split action is conjugate to an idempotent split action over the vertex projections $F^u:C\rightarrow \BasedBim(\KQ)_{I_v}$. This process is invertible since conjugation is invertible so we can go back as well.
\end{proof}

For the rest of the paper we will consider monoidal functor $F:C\rightarrow \BasedBim(\KQ)_{I_v}$ since all idempotent split based actions are conjugate to these ones and all separable based actions correspond to idempotent split ones.

\section{\texorpdfstring{\centering  Endofunctor Subcategories }{Endofunctor Subcategories}}

In this section, we will study two important subcategories of endofunctor categories. We will then classify monoidal functors from $C$ into those categories in some nice classes of examples.

We define the endofunctor category $\End_{(\widetilde{Q},m,i)}(\VecC(M))$. We will need to understand and use $\VecC(M)$ the category of $M$ graded vector spaces of a finitely semisimple linear category $M$ to achieve this. 

\begin{define}[From \protect\cite{JP17}] $\VecC(M)$ is  defined to be the category of linear functors $F:M^{op}\rightarrow \VecC$ whose morphisms are natural transformations of these functors. 
    
\end{define}

 Conceptually, we think of this category as allowing there to be infinite direct sums of simple objects as opposed to only having finite direct sums in $M.$ This category is equivalent to $\Ind(M)$ the Ind completion of $M$. We remark that this description is only this nice in the case where $M$ is a finitely semisimple linear category. Otherwise, the Ind completion and the cocompletion are not necessarily equivalent. We use $\VecC(M)$ to now define two important endofunctor subcategories, $\End_{\widetilde{Q}}(\VecC(M))$ and $\End_{(\widetilde{Q},m,i)}(\VecC(M))$. Remember, given a quiver, $Q$ we have an endofunctor corresponding to it. Then endofunctor $\widetilde{Q}$ corresponds to the quiver generated by a quiver $Q$ that comprises all paths of any length in $Q$.

\begin{define}
    We define the monoidal category $\End_{\widetilde{Q}}(\VecC(M))=\{(F,\phi_F)|F\in \End(\VecC(M)),\hspace{1mm} \\\phi_F:\widetilde{Q}\circ F\xrightarrow{\sim}F\circ \widetilde{Q} \}$ 
    where morphisms are natural transformations that commute with the half braiding, that is $f\in\Hom(F,G)$ such that $\phi_G\circ(f\otimes Id_{\widetilde{Q}})=(Id_{\widetilde{Q}}\otimes f)\circ \phi_F$.
\end{define}

This is a monoidal category. We shall see later  that there is a nice classification of monoidal functors $F: C\rightarrow \End_{\widetilde{Q}}(\VecC(M))$. Now we define a subcategory $\End_{(\widetilde{Q},m,i)}(\VecC(M))$ where $(\widetilde{Q},m,i)$ is a monad structure on $\widetilde{Q}$.

\begin{define}
We define the subcategory $\End_{(\widetilde{Q},m,i)}(\VecC(M))$ as,
   \[
    \End_{(\widetilde{Q},m,i)}(\VecC(M))=\{(F,\phi_F)\in \End_{\widetilde{Q}}(\VecC(M))\}\]
    such that the following properties are satisfied for each $(F,\phi_F)$,
    \begin{enumerate}
    \item $\phi_F \circ(m\otimes \id_F)=(\id_F \otimes m)\circ(\phi_F\otimes\id_{\widetilde{Q}})\circ(\id_{\widetilde{Q}}\otimes \phi_F)$. Where $m$ is a natural transformation $m:\widetilde{Q} \circ \widetilde{Q} \rightarrow \widetilde{Q}$ that corresponds to multiplication in the monad $\widetilde{Q}$.  
        \item $\phi_{F}\circ(i\otimes \id_F)=(\id_F \otimes i) \circ \phi^{\Id}_{F}$. Where $i:\Id_M\rightarrow \widetilde{Q}$ is the embedding of the identity functor in, $\widetilde{Q}$ and $\phi^{\Id}_{F}$ is the trivial half braiding with the identity.
    \end{enumerate}
  
\end{define}
\begin{remark}
    $(\widetilde{Q},m,i)$ can be thought of a choice of monad structure on $\widetilde{Q}$ generated by a finite subfunctor $Q$ and the identity $\Id$. The category $\End_{(\widetilde{Q},m,i)}(\VecC(M))$ is the collection of endofunctors who's half braidings are compatible with the monad structure.
\end{remark}

\begin{lem}
    $\End_{(\widetilde{Q},m,i)}(\VecC(M))$ is a full subcategory of $\End_{\widetilde{Q}}(\VecC(M))$.
\end{lem}

\begin{proof}
We need to show that $I:\End_{(\widetilde{Q},m,i)}(\VecC(M)) \rightarrow \End_{\widetilde{Q}}(\VecC(M)) $ is a fully faithful functor. This follows immediately because the data of an object in $\End_{(\widetilde{Q},m,i)}(\VecC(M))$ is still $(F,\phi_F)$ but $\phi_F$ the half braiding is compatible with $m$ and $i$. The morphisms between objects are unchanged. Thus, $\End_{(\widetilde{Q},m,i)}(\VecC(M))$ is a full subcategory of $\End_{(\widetilde{Q},m,i)}(\VecC(M))$.
\end{proof}

We are interested in studing monoidal functors from $C \rightarrow \End_{\widetilde{Q}}(\VecC(M))$ and what categorical data classifies them. This will be important later when we connect these categories to based bimodules of path algebras.

\begin{thm}\label{thm: classification of functors into End(Vec(M))}
    A monoidal functor $F:C\rightarrow \End_{\widetilde{Q}}(\VecC(M))$ is equivalent to the data of a $C$ module category structure on $\VecC(M)$ and a $C$ module endofunctor $(\widetilde{Q},_{-})$ in $\End_C(\VecC(M))$.
    
\end{thm}
\begin{proof}
    This is just a direct application of the bijection between monoidal functors $F:C\rightarrow \End(M)$ and module category structure on $M$ \protect\cite{EGNO16} as well as realizing that $\widetilde{Q}$ is precisely a dual functor.
\end{proof}
\begin{remark}
    A tensor functor $F:C\rightarrow \End(\VecC(M))$ actually produces a $C$ module structure on just $M$. This is because we can identify endofunctors with a $\Irr(M)\times \Irr(M)$ matrix of vector spaces $V_{ij}$. Since $C$ is a fusion category all images $F(X)$ are dualizable functors, then $V_{ij}$ must all be finite dimensional and thus $F(X)$ is actually an endofunctor on $M$, therefore producing a module category structure on $M$.
\end{remark}

Given two monoidal functors $F:C\rightarrow \End_{\widetilde{Q}}(\VecC(M))$ and $G:C\rightarrow\End_{\widetilde{Q}}(\VecC(M))$ we want to understand when they are equivalent up to monoidal natural isomorphism.

\begin{lem} \label{lem:C mod functor and mon nat iso}
 Given monoidal functors $F,G:C\rightarrow \End(M)$ are monoidally naturally isomorphic iff there exists a $C$ module functor $(\Id,\rho)$  such that $\Id\circ F_X \xrightarrow{\rho_X}G_X\circ \Id$ for all objects $X$ in $C$.
\end{lem}

\begin{proof}
    We shall first construct the morphism and then show it satisfies the necessary data. Given $\psi:F\rightarrow G$ monoidal natural isomorphism we can construct $\overline{\psi}: \Id\circ F_X \rightarrow G_X\circ \Id$ by defining $\overline{\psi}:=  r_x^{-1}\circ\psi_X\circ l_x:\Id\circ F_X\xrightarrow[]{\sim}G_X\circ \Id.$

    Here $l_x$ and $r_x$ here are the usual left and right unitor natural isomorphisms. Conversely, if we had a dual endofunctor $(\Id,\overline{\psi})$ we could define monoidal natural isomorphism $\psi$ by inverting both the unitors. It's clear that both $\psi$ and $\overline{\psi}$ are natural isomorphism with the desired target and range. Now to show that they satisfy the desired extra data consider the following diagram.

\[\adjustbox{scale=.65,center}{\begin{tikzcd}
	&&&& {\Id\circ F_{X\otimes Y}(M)} & 10 \\
	\\
	\\
	&& 1 && {F_{X\otimes Y}(M)} & 2 \\
	{\Id\circ(F_x\circ(F_Y(M))} &&& {F_X\circ F_Y(M)} & 3 & {G_{X\otimes Y}(M)} &&& {G_{X\otimes Y}\circ \Id(M)} \\
	& 4 && 5 & {G_X\circ G_Y(M)} \\
	&&& {} &&& 9 \\
	& {G_X\circ F_Y} && 6 \\
	&& 7 && {G_X\circ F_Y} \\
	&&&& 8 \\
	{G_X\circ(\Id\circ F_Y(M))} &&&&&&&& {(G_X\circ G_Y)\circ \Id(M)}
	\arrow["{l_{F_{X\otimes Y}}}"{description}, from=1-5, to=4-5]
	\arrow["{\Id(J_{X,Y}^F)}"', from=1-5, to=5-1]
	\arrow["{\overline{\psi_{X\otimes Y}}}", from=1-5, to=5-9]
	\arrow["{J_{X,Y}^F}"', from=4-5, to=5-4]
	\arrow["{\psi_{X\otimes Y}}", from=4-5, to=5-6]
	\arrow["{l_{F_X}\otimes\Id_Y}"{description}, from=5-1, to=5-4]
	\arrow["{\overline{\psi_X}\otimes\Id_Y}"', from=5-1, to=11-1]
	\arrow["{\psi_X\otimes \psi_Y}", from=5-4, to=6-5]
	\arrow["{\psi_{F_X}\otimes\id_{F_Y}}"{description}, from=5-4, to=8-2]
	\arrow["{r^{-1}_{G_{X\otimes Y}}}"', from=5-6, to=5-9]
	\arrow["{J_{X,Y}^G}", from=5-6, to=6-5]
	\arrow["{J_{X,Y}^G\otimes \Id_{\Id}}", from=5-9, to=11-9]
	\arrow["{\id_X\otimes r_{G_Y}^{-1}}", from=6-5, to=11-9]
	\arrow["{\Id_{G_X}\otimes \psi_Y}"{description}, from=8-2, to=6-5]
	\arrow["\Id"{description}, from=8-2, to=9-5]
	\arrow["{\Id_{G_X}\otimes \psi_Y}"{description}, from=9-5, to=6-5]
	\arrow["{r_{G_X}\otimes \Id_{F_Y}}"{description}, from=11-1, to=8-2]
	\arrow["{\Id_{G_X}\otimes l_{F_Y}}"{description}, from=11-1, to=9-5]
	\arrow["{\Id_X\otimes \overline{\psi_Y}}"', from=11-1, to=11-9]
\end{tikzcd}}\]
Diagrams (1,9) commute by naturality.
Diagrams (2,4,8) commutes by definition of $\psi$ and $\overline{\psi}$. Diagram 5 commutes by the tensorator. Diagram 6 commutes trivially. Diagram 7 is the triangle diagram in the catgegory of endofunctors. That leaves us with diagram 3 and diagram 10. This corresponds to the definition of monoidal natural isomorphism for 3 and $C$ module functor for 10. So assuming one implies the other will commute and vice versa thus giving us a bijection.
\end{proof}

\begin{prop}\label{prop: equiv of C mod and mon nat iso1}
    Given monoidal functors $F,G:C\rightarrow \End_{\widetilde{Q}}(\VecC(M))$ if,
    
    \begin{enumerate}
        \item $F,G$ both equip $M$ with the same $C$ module structure up to the identity $C$ module functor $\Id$ and a half braiding $\rho$.
        \item There is an isomorphism of $C$ module endofunctors $(\Id,\rho)\circ (\widetilde{Q}^F,\phi^F_{\widetilde{Q}})\circ (\Id,\rho)^{-1}\cong (\widetilde{Q}^G,\phi^G_{\widetilde{Q}}) $.
    \end{enumerate}
Then $F$ and $G$ monoidally naturally isomorphic.
\end{prop}
\begin{proof}
    Suppose that $F$ and $G$ generate $C$ module categories $M_F$ and $M_G$ such that $(\Id,\phi):M_F\rightarrow M_G$ is an equivalence of $C$ module categories and $(\Id,\rho)\circ (\widetilde{Q}^F,\phi^F_{\widetilde{Q}})\circ (\Id,\rho)^{-1}\cong (\widetilde{Q}^G,\phi^G_{\widetilde{Q}})$. By Lemma \protect\ref{lem:C mod functor and mon nat iso}  we can construct a monoidal natural isomorphism between $F$ and $G$ when thought of as functors into $\End(\VecC(M))$. This natural isomorphism lifts to $\End_{\widetilde{Q}}(\VecC(M))$ since  $(\Id,\rho)\circ (\widetilde{Q}^F,\phi^F_{\widetilde{Q}})\circ (\Id,\rho)^{-1}\cong (\widetilde{Q}^G,\phi^G_{\widetilde{Q}})$. 
\end{proof}

\begin{prop}\label{prop:equiv of C mod and mon nat iso2}
    If $F,G:C\rightarrow \End_{\widetilde{Q}}(\VecC(M))$ are monoidally naturally isomorphic then,
    \begin{enumerate}
        \item $F,G$ both equip $M$ with the same $C$ module structure up to the identity $C$ module functor $\Id$ and a half braiding $\rho$.
        \item There is an isomorphism of $C$ module endofunctors $(\Id,\rho)\circ (\widetilde{Q}^F,\phi^F_{\widetilde{Q}})\circ (\Id,\rho)^{-1}\cong (\widetilde{Q}^G,\phi^G_{\widetilde{Q}}) $.
    \end{enumerate}
\end{prop}
\begin{proof}
    Suppose there exists a monoidal natural isomorphism $\psi:F\rightarrow G$. Then $F_X\xrightarrow[]{\psi_X}G_X$ is a natural isomorphism of functors for isomorphism classes of objects $X \in C$. By Lemma \protect\ref{lem:C mod functor and mon nat iso} this data is equivalence of module categories $M_G$ and $M_F$ up to an invertible $C$ module functor $(\Id,\rho)$. It follows that the dual functors $(\widetilde{Q}^F,\phi^F_{\widetilde{Q}})$ and $(\widetilde{Q}^G,\phi^G_{\widetilde{Q}})$ are equivalent up to conjugacy by that invertible $C$ module functor, $(\Id,\rho)\circ (\widetilde{Q}^F,\phi^F_{\widetilde{Q}})\circ (\Id,\rho)^{-1}\cong (\widetilde{Q}^G,\phi^G_{\widetilde{Q}})$. 
\end{proof}

\begin{thm}\label{thm:based action bijection}
 Fix a fusion category $C$. Isomorphism classes of monoidal functors $F:C\rightarrow \End_{\widetilde{Q}}(\VecC(M))$ up to monoidal natural isomorphism are in bijection with

   \begin{enumerate}
        \item A choice of $C$ module category structure on $M$ up to invertible $C$ module functors of the form $(\Id,\rho)$.
        \item A conjugacy class of objects $[(\widetilde{Q},\phi_{-})]$ in $\End_C(\VecC(M))$ up to conjugation by $C$ module endofunctors of the form $(\Id,\rho)$. 
   \end{enumerate} 
\end{thm}
\begin{proof}

This bijection follows from Propositions \protect\ref{prop: equiv of C mod and mon nat iso1}, \protect\ref{prop:equiv of C mod and mon nat iso2}.
\end{proof}

\begin{remark}
    This is very similar data to the data that classifies tensor algebras in \protect\cite{EKW21}. The main difference here is the fineness of our equivalences. Monoidal natural isomorphism allows for less flexibility than classification up to Morita equivalence of module categories and module endofunctors. 
\end{remark}

For algebra complete fusion categories, we can say with more certainty what the module categories are up to monoidal natural isomorphism. For example for $\Rep(S_3)$, a non algebra complete fusion category there are indecomposable two module categories structures on a linear semisimple category with $3$ simple objects, $\Rep(S_3)$ and $\Rep(A_3)$. Conversely, for algebra complete fusion categories each indecomposable semisimple module category is equivalent to $C$.

\begin{cor}\label{cor:algebra complete mon nat iso}
Fix an algebra complete fusion category $C$. Isomorphism classes of monoidal functors $F:C\rightarrow \End_{\widetilde{Q}}(\VecC(M))$ up to monoidal natural isomorphism are in bijection with

   \begin{enumerate}
        \item A choice of family of bijections $\chi_i:\Irr(M_i)\rightarrow \Irr(M_i)$.
        \item An isomorphism class of dual functors $(\widetilde{Q},\phi_{-})$ in $\End_C(\VecC(M))$. 
   \end{enumerate} 
\end{cor}
\begin{proof}
By Theorem \protect\ref{thm:based action bijection} monoidal functors $F$ up to monoidal natural isomorphism are classified by an equivalence class of $C$ module categories up a dual functor of the form $(\Id,\rho)$ and a conjugacy class $[\widetilde{Q},\phi]$. Notice that for an algebra complete fusion category since all semisimple module categories are of the form $C^{\oplus n}$ and $\End_C(C)\cong C^{op}$, then the only $C$ module functor of the form $(\Id,\rho)$ is $(\Id,\rho_{\Id})$. Suppose $M$ is a semisimple module category such that $M\cong \oplus_{i=1}^n M_i$ where $M_i$ is an indecomposable semisimple module category. Then, a choice of $C$ module category corresponds to a labeling of the simple objects from each $M_i$ with the simple objects in $C$. This corresponds to a family of bijections $\{\chi_i:\Irr(M_i)\rightarrow \Irr(M_i)\}_{i=1}^n$. Now our conjugacy class of $C$ module endofunctors $[(\widetilde{Q},\phi)]$ just reduces to an isomorphism class of $C$ module endofunctors $(\widetilde{Q},\phi)$, since the only invertible $C$ module endofunctor is of the form $(\Id,\rho_{\Id})$.
\end{proof}

\begin{remark}
    Our goal later in the paper will be to figure out which monoidal functor $F:C\rightarrow \End_{\widetilde{Q}}(\VecC(M))$ are actually monoidal functors into $\End_{(\widetilde{Q},m,i)}(\VecC(M))$. Since $\End_{(\widetilde{Q},m,i)}(\VecC(M))$ is a full subcategory of $\End_{\widetilde{Q}}(\VecC(M))$ we can determine this by checking if they satisfy the required two properties.
\end{remark}

To better understand the category $\End_{\widetilde{Q}}(\VecC(M))$ and what monoidal functors from $C\rightarrow \End_{\widetilde{Q}}(\VecC(M))$ exist, we must first understand the dual functor $\widetilde{Q}$. 
For a semisimple module category, $M$ then $\End_C(M)$ is a multifusion category with simple objects $\{Q_1,...,Q_m\}$. The category $\VecC(\End_C(M))$ can be thought of allowing infinite sums of these simple $C$-Module endofunctors. We want to understand the relationship between $\End_C(\VecC(M))$ and the category $\VecC(\End_C(M))$. 

\begin{thm}\label{thm:dual cat and filtered colim}
Let $C$ be a fusion category and fix a semisimple $C$ module category structure on $M$. There is an equivalence of categories between $\End_C(\VecC(M))$ and $\VecC(\End_C(M))$. 
\end{thm}
\begin{proof}
The dual category $\End_C(M)$ is equivalent to the category $C^T$ of algebras over the monad $T:C\rightarrow C$ that sends $T(X)\rightarrow A\otimes X \otimes A$. The modules over this monad are $A$ bimodules which are equivalent as categories  to the dual category $\End_C(M)$. Thus, it suffices to show $\VecC(C^T)\cong (\VecC(C))^T$. There is a fully faithful functor $F:\VecC(C^T)\rightarrow \VecC(C)^T$ that forgets the $T$ algebra structure and then remembers it. That is, this functor takes an element in $\VecC(C^T)$ which is a functor $G:(C^T)^{op}\rightarrow \VecC$ and forgets the $T$ algebra structure to produce a functor, $G':C^{op}\rightarrow \VecC$ then re-equips $G'$ with a $T$ algebra structure. We will show that $G$ is essentially surjective. Let $V$ be an object in $\VecC(C)^T$ that is $V:C^{op}\rightarrow \VecC$ with the structure of a $T$ algebra. Take a compact subobject of $V$ denoted $W$, by compact subobject we mean a subfunctor $W:C^{op}\rightarrow \VecC$ such that $\dim(W(X))$ is finite dimensional for all simple $X$. We can construct a $T$ algebra containing $W$ by considering $A\otimes W \otimes A$ denoted $\widetilde{W}$. Notice that $A\otimes W\otimes A$ is a compact object in $\VecC(C)$ since $W$ and $A$ are dualizable hence compact. By construction $\widetilde{W}$ has a $T$-algebra structure which implies that it's a compact object in $\VecC(C^T)$.

Now, it suffices to show that the embedding of $A\otimes W\otimes A\rightarrow A\otimes V \otimes A\rightarrow V$ is a compact subobject of $V$. Note that if $V$ is compact, this follows automatically. $V$ is compact when $V$ is an object in $\End_C(M)$. Thus, we consider when $V$ is not compact. Since $\widetilde{W}$ is a functor from $C\rightarrow \VecC$ it implies that it can be represented as a row vector of vector spaces $(\widetilde{W}_{X_1},....,\widetilde{W}_{X_n})_{X_i\in \Irr(C)}$ where each $\widetilde{W}_{X_i}$ is finite dimensional. Now $V$ can be represented as a row vector $(V_{X_1},...,V_{X_n})_{X_i\in \Irr(C)}$ where $V_i$ is a vector space of any dimension. A morphism between these two objects is a natural transformation which can be represented as a row vector of linear transformations
$(\phi_{X_1},...,\phi_{X_n})$ where $\phi_{X_i}:\widetilde{W}_{X_i}\rightarrow V_{X_i}$ is a linear transformation. The image of a finite dimensional vector space is finite dimensional, thus the image of $\widetilde{W}$ is a compact subobject.

Now consider a chain of compact subobjects $W_i\subset V$ such that the filtered colimit of the $W_i$'s is $V$. Then we can construct a chain of $T$ algebras $\widetilde{W_i}\subset V$. This chain of $T$ algebras has a filtered colimit isomorphic to $V$. Thus $V$ is in $\VecC(C^T)$ and $F(V)\cong V$. Thus $F$ is essentially surjective and we have that
$\End_C(\VecC(M))\cong\VecC(\End_C(M))$.
\end{proof}

\begin{cor}
    $\widetilde{Q}$ decomposes as a direct sum of simple $\End_C(M)$ module functors, that is $\widetilde{Q} \cong \bigoplus_{Q_i} n_iG_i,\hspace{1 mm} n_i \in \Z_{\geq 0} \cup \{\infty\},\hspace{1 mm} G_i\in \Irr(\End_C(M))$.
\end{cor}

\begin{proof}
     Since $C$ is fusion, then $\End_C(M)$ is multifusion and thus has simple objects $G_i$ and every element is a direct sum of those simple objects $G_i$. By Theorem \protect\ref{thm:dual cat and filtered colim} it follows that $\End_C(\VecC(M)) \cong \VecC(\End_C(M)) $ and thus we conclude that the objects in $\End_C(\VecC(M))$ are possibly infinite direct sums of simple objects in $\End_C(M)$.
\end{proof}

 We will now use this to classify the number of monoidal functor $F:C\rightarrow \End_{\widetilde{Q}}(\VecC(M))$ when $Q$ is a quiver with strongly connected components. A quiver is strongly connected if for each vertex $a$ and $b$ there is a path from $a$ to $b$ and a path from $b$ to $a$. Monoidal functors from $C$ into $\End_{\widetilde{Q}}(\VecC(M))$ correspond to the different decompositions of $\widetilde{Q}$ into simple objects. We shall see for these quivers, there are either $1,$ or $\infty$ of these decompositions, and thus there are either $0,1$ or $\infty$ monoidal functors from $C\rightarrow  \End_{\widetilde{Q}}(\VecC(M))$.

\begin{thm} \label{thm: counting TFEnd_Q}
Let $Q$ be a quiver with strongly connected components. Let $C$ be a fusion category and fix a semisimple $C$ module category structure on $M$. Assuming there is a monoidal functor $F:C\rightarrow \End_{\widetilde{Q}}(\VecC(M))$, then there is only 1 monoidal functor up to monoidal natural isomorphism iff for all $G \in \Irr(\End_C(M))$ that appears in the decomposition of $\widetilde{Q}$ there exists $M_i,M_j \in M$ such that $G$ is the only simple such that $\dim(\Hom(M_j,G(M_i))>0$.
\end{thm}

\begin{proof}
  Suppose for all simple $G_k \in \End_C(M)$ that appear in the decomposition of $\widetilde{Q}$ there exists $M_i,M_j \in M$ such that $G$ is the only simple such that $\dim(\Hom(M_j,G(M_i))>0$. Then if $\widetilde{Q}\cong\sum n_kG_k$ it follows that $G_k$ is the only simple dual functor such that $\dim(\Hom(M_i,G_k(M_j))>0$ for some $i,j$. Then we have a unique scalar $n_k$ in an entry of the matrix representation of our functor. Since $Q$ is strongly connected, then $\widetilde{Q}$ has an infinity in each non-zero entry which implies that each $n_k$ must be $\infty$. Therefore, there is only one decomposition and thus one action.

  Conversely, suppose that there is one monoidal functor $F:C\rightarrow \End_{\widetilde{Q}}(\VecC(M))$. Suppose there existed a $G\in \Irr(\End_C(M))$ such that for all $M_i,M_j \in M$ if $\dim(\Hom(M_j,G(M_i))>0$ then there exists some $G'$ such that $\dim(\Hom(M_j,G'(M_i))>0$. Since $Q$ is a quiver with strongly connected components, it follows that each non-zero edge in $\widetilde{Q}$ has infinite edges. This implies that if $G$ was removed from $\sum n_kG_k$ then $\widetilde{Q}$ would remain unchanged. Thus, we get a different monoidal functor for each coefficient $n_g$ which is a contradiction. 
\end{proof}

\begin{cor}\label{cor:actions on EndQ}
    Let $Q$ be a quiver with strongly connected components. Let $C$ be a fusion category and fix a semisimple $C$ module category structure on $M$. If there exists a monoidal functor $F:C\rightarrow \End_{\widetilde{Q}}(\VecC(M))$ then there are up to monoidal natural isomorphism there are either $1$ or $\infty$ such functors.
\end{cor}

\begin{proof}
    Suppose that $Q$ has more than one monoidal functor. Then this implies that there existed a $G\in \Irr(\End_C(M))$ such that for all $M_i,M_j \in M$ if $\dim(\Hom(M_j,G(M_i))>0$ then there exists some $G'$ such that $\dim(\Hom(M_j,G'(M_i))>0$. Then in the decomposition, $\widetilde{Q}=\sum n_kG_k$ the coefficient on $G$ can be any element in $\Z_{\geq 0}\cup \infty$, which corresponds to a different monoidal functor from $C$ into $\End_{\widetilde{Q}}(\VecC(M))$. Thus, if there are more than one action, there are infinite.
\end{proof} 

If there is no module category structure on $M$ then there will be no monoidal functors  $F:C \rightarrow \End_{\widetilde{Q}}(\VecC(M))$. There are nice examples that we can construct for each of the three cases  $F:C\rightarrow \End_{\widetilde{Q}}(\VecC(M))$. We describe one of each of them below,

\begin{example}
    For an example of a category with $0$ monoidal functors in $\End_{\widetilde{Q}}(\VecC(M))$ consider any fusion category that doesn't have a fiber functor to $M\cong\VecC_{f.d}$ the category of finite dimensional vector spaces. It follows that we have no monoidal functors of $C$ into $\End_{\widetilde{Q}}(\VecC(M))$, since there are no semisimple module categories of $C$ with one  isomorphism class of simple objects. 
\end{example}

\begin{example}{\label{ex:one action}}
    We can construct an example of $\Fib$ acting on $\Fib^{op}$ then one can compute that $\End_{\Fib}(\Fib)\cong \Fib^{op}$ and thus, $\widetilde{Q} \cong  n_\1\1\oplus n_{\tau}\tau$. Notice that by Theorem \protect\ref{thm: counting TFEnd_Q} that for $\tau$ and $\1$ we have that $\tau$ is the only dual functor taking $\1$ to $\tau$ then it follows that $n_i=\infty$ for all $i$ implying that there is only one decomposition of $\widetilde{Q}$ and thus only one action. 
\end{example}

\begin{example} 
For a case of infinite actions, take any non-trivial fusion category $C$ that has a fiber functor onto $\VecC$. An example of this is $C\cong\VecC(G)$ acting on $\VecC$ via the forgetful functor. The dual category here is equivalent to $\Rep(G)$. Each $C$ module endofunctor can be determined by the image of $\K$. We can construct that infinite number of combinations of the dual objects to produce $\widetilde{Q}$ giving us an infinite amount of actions. 
\end{example}

\begin{remark}
    It's also possible to have no monoidal functors $F:C\rightarrow\End_{\widetilde{Q}}(\VecC(M))$ when $\widetilde{Q}$ doesn't have a decomposition into a sum of simple dual functors. 
\end{remark}

For $\End_{(\widetilde{Q},m,i)}(\VecC(M))$ we produce the following corollary about the number of monoidal functors from $C$ to $\End_{(\widetilde{Q},m,i)}(\VecC(M))$ when $Q$ is quiver with strongly connected components. 

\begin{cor} \label{cor:either 0, at most 1 or at most infinty}
Let $Q$ be a quiver with strongly connected components. For a fixed semisimple $C$ module category structure on $M$, there are either at most $1$ or at most infinite monoidal functors from $C$ to $\End_{(\widetilde{Q},m,i)}(\VecC(M))$ up to monoidal natural isomorphism.
\end{cor}

Algebra complete fusion categories have nice properties to them that allow us to say how many monoidal $F:C\rightarrow \End_{\widetilde{Q}}(\VecC(M))$ there are regardless of the quiver structure.
\begin{lem}\label{lem:actions of morita triv}
    Fix an algebra complete fusion category $C$ and fix an indecomposable module category structure on $M$, if there exists a monoidal functor $F:C\rightarrow \End_{\widetilde{Q}}(\VecC(M))$, then $F$ is unique.
\end{lem}
\begin{proof}
  Note that since $M$ is indecomposable, then $M\cong C$ as $C$ module categories. Fix an algebra complete category $C$ and module category $M$, then if there exists a monoidal functor $F:C\rightarrow \End_{\widetilde{Q}}(\VecC(M))$ up to monoidal natural isomorphism we must choose a bijection of $\Irr(M)$ and a conjugacy class of $[\widetilde{Q},\phi]$. Then by Theorem \protect\ref{thm:dual cat and filtered colim} it follows that $\widetilde{Q}\cong Z$ for some isomorphism class of objects $Z \in C$. Monoidal functors up to monoidal natural isomorphism correspond to different decompositions of $\widetilde{Q}$ where all simple endofunctors are $F_X$ which denotes the action of isomorphism class of simple object $X$. Notice that $F_X(\1)\cong X$ for all $X\in \Irr(C)$ and $F_X$ is the only functor that takes $\1$ to $X$. It then follows that decomposition of $\widetilde{Q}$ is unique. Thus, there is only one monoidal functor up to monoidal natural isomorphism.
\end{proof}
\begin{cor}\label{cor:alg complete actions}
    Let $C$ be an algebra complete fusion category and fix a semisimple module category structure on $M$. If there exists a monoidal functor $F:C\rightarrow \End_{\widetilde{Q}}(\VecC(M))$, then $F$ is unique.
\end{cor}
\begin{proof}
    Consider a quiver $Q$ such that there exists a monoidal functor $F:C\rightarrow \End_{\widetilde{Q}}(\VecC(M))$. Since $M\cong C^{\oplus n}$, $\widetilde{Q}$ can be represented by an $n\Irr(C)\times n\Irr(C)$ block matrix where each block contains an element in $\End_C(\VecC(M))$. Given an element in $\End_C(\VecC(M))$ by Lemma \protect\ref{lem:actions of morita triv} there is only one decomposition of each $\widetilde{Q}_{ij}$ into simple objects. Thus, up to monoidal natural isomorphism there is one monoidal functor $F:C\rightarrow \End_{\widetilde{Q}}(\VecC(M))$. 
\end{proof}

\section{\texorpdfstring{\centering  Based Bimodules of Path Algebras and Endofuntors}{The  Based Bimodules of Path Algebras and Endofuntors}}
 In this section, we describe an equivalence of monoidal categories between $\End_{(\widetilde{Q},m,i)}(\VecC(M))$ and $\BasedBim(\KQ)_{I_v}$. We have previously shown that monoidal functors into $\End_{\widetilde{Q}}(\VecC(M))$ and $\End_{(\widetilde{Q},m,i)}(\VecC(M))$ can be expressed in terms of $C$ module categories. We want to use this results to classify based actions of fusion categories on path algebras. We first shall prove a couple of lemmas.

 \begin{lem} \label{lem: elements in V_X}
     Let $(X,V_X)$ be a based bimodule in $\BasedBim(\KQ)_{I_v}$. Then for all $z\in V_X$, $z=\sum_i c_ix_ip_{v_i}$ where $x_i$ are the elements in the projective basis of $V$ and $p_{v_i}$ is a vertex in $V_{\KQ}$. 
 \end{lem}
 \begin{proof}
    Follows from Corollary \protect\ref{cor: coev lives in V_1}
 \end{proof}

\begin{lem} \label{lem:PBKQ}
    Let $\{x_i\}$ be a right projective basis for $X$ a right module of $\K Q$. Then $\{x_ip_v\}_{i,v}$ is a right projective basis for $X$.
\end{lem}

\begin{proof}
    Since $\{x_i\}$ is a right projective basis, then for all $x \in X$ we have that

    \[x=\sum_{i=1}^n x_ix_i^*(x) \]

    where $x_i^*:X_{\KQ}\rightarrow \KQ_{\KQ}$.

    Now, each $x_i$ can be broken down into a $x_ip_v$. Consider,

    \[\sum_{i=1,v}^n x_ip_v(x_ip_v)^*(x)\]

    where we define $(x_ip_v)^*$ as the projection $\pi_v$ composed with $(x_i)^*$. $\pi_v:\KQ\rightarrow \KQ$ is the right module homomorphism that sends $\alpha$ to $p_v\alpha$. Thus,

 \[x=\sum_{i=1,v}^n x_ip_v(\pi_v\circ (x_i)^*)(x).\]

 which implies that $\{x_ip_v\}_{i,v}$ is a projective basis of $X$.
\end{proof}
Now consider the base space $V_X$. Then by above it follows that for each $z\in V_X$ we have that for some $i,v$ that, 
\[z=\sum_{i=1,v}^n x_ip_v.\] 
Since $z$ can be written as a sum of objects in our projective basis it follows that there is a half braiding of element $\alpha \in \KQ$ and $V_X$,

\[\alpha z=\alpha \sum_{i=1,v}^n x_ip_v=\sum_{j,u}x_jp_u(x_jp_u)^*(\alpha \sum_{i=1,v}^n x_ip_v).\]

Noting that $ x_jp_u=z_{j,u}$ for some element $z_{i,k}\in V_X$ this implies,

\[\alpha z=\sum_{i=1,v}^n z_{i,v}\beta \hspace{2mm } \beta \in \KQ.\]
Since all our bimodules are dualizable, this defines a bijection between $\KQ V_X$ and $V_X\KQ$ as sets that is compatible with multiplication and the unit in $\KQ$. We will use this fact throughout this section.

 Now we will state and prove our main theoretical result of the paper. This result is what is going to allow us to understand and classify monoidal functors from $C$ into $\BasedBim(\KQ)_{I_v}$.
 
\begin{thm}\label{thm:TE BimEnd}
    There is an equivalence of monoidal categories between $\BasedBim(\KQ)_{I_v}$ and  
 $\\\End_{(\widetilde{Q},m,i)}(\VecC(M))$, where $M$ is the semisimple category equivalent to $\Mod(\KQ_0)$.
\end{thm}

\begin{proof}
    A monoidal functor is a pair $(G,\epsilon_{-,-})$ where $G$ is a functor and $\epsilon_{-,-}$ is a natural isomorphism. We will show that $G$ gives equivalence of categories and then define $\epsilon_{-,-}$. We define $G$ as follows,

    \[G:\BasedBim(\KQ)_{I_v}\rightarrow\End_{(\widetilde{Q},m,i)}(\VecC(M)). \]

    \[G((X,V_X)):= (F_{V_X},\phi_{F_X}).\]

    Where $F_{V_X}(v)=\bigoplus n_iw_i$ such that $n_i=\dim(p_{w_i}V_X p_v)$. We define our half braiding by using the half braiding of $V_X$ with $\KQ$. That is $\KQ*V_X=V_X*\KQ$ translates to $\phi_F:\widetilde{Q}\circ F_{V_X}\xrightarrow{\sim} F_{V_X}\circ 
 \widetilde{Q}$. Since our bimodule half braiding is compatible with multiplication and the projections, it follows that $\phi_F$ will be a half braiding in $\End_{(\widetilde{Q},m,i)}(\VecC(M))$. We will show that $G$ is fully faithful and essentially surjective.

   We will first show  that $G$ is essentially surjective. Consider $(F,\phi_F)\in \End_{(\widetilde{Q},m,i)}(\VecC(M))$ that is an endofunctor $F:M\rightarrow M$ and a natural isomorphism $\phi_F:\widetilde{Q}\circ F\rightarrow F\circ\widetilde{Q}$. We want to find a based $\KQ$ bimodule $(X,V_X)$ such that $G((X,V_X))\cong (F,\phi_F)$. We can construct a bimodule $(X,V_X)$ such that $\dim(p_vV_Xp_w)$ is equal to $\dim(\Hom(Y_w,F(Y_v))$ and who's half braiding with $\KQ$ is defined exactly how $\phi_F:\widetilde{Q}\circ F\xrightarrow{\sim} F\circ \widetilde{Q}$ is defined. This will be a $\KQ$ bimodule since $\phi_F$ is compatible with multiplication and the unit. Thus, $G(X,V_X)\cong (F,\phi_F)$ and $G$ is essentially surjective.  

 To see that $G$ is faithful, let $f:(X,V_X)\rightarrow (Y,V_Y)$ be a $\KQ$ bimodule homomorphism such that $G(f)=0$ the zero natural transformation. That is $G(f)$ is a natural transformation from $F_X\rightarrow F_Y$. This implies that $g$ can be represented by a $|\Irr(M)| \times |\Irr(M)|$ matrix of linear transformations. Since the linear map  is zero in each entry, it follows that $f$ must also be zero in each entry.

 Let $g \in \Hom(G(X, V_X),G(Y,V_Y))$. That is $g$ is a natural transformation from $F_{V_X}\rightarrow F_{V_Y}$ that is compatible with the half braiding of $\widetilde{Q}$. This implies that $g$ can be represented by a $|\Irr(M)| \times |\Irr(M)|$ matrix where each entry is a linear transformation  $g_{ji}$ satisfying the property $\phi_{F_Y}\circ (g\otimes \id_{\widetilde{Q}})=(\id_{\widetilde{Q}}\otimes g)\circ \phi_F$. Then we define $f$ such that $f$ is a $\BasedBim(\KQ)_{I_v}$ homomorphism whose component linear maps $f_{ji}$ corresponds to the morphism $g_{ji}$. Since $f$ is a $\KQ$ bimodules homomorphism, it follows that $\phi_{F_Y}\circ (G(f)\otimes \id_{\widetilde{Q}})=(\id_{\widetilde{Q}}\otimes \hspace{1mm} G(f))\circ \phi_F$. In particular, the property we are using is \[\leftidx_a\alpha_bf(\leftidx_bx_c)=f(\leftidx_a\alpha_b\leftidx_bx_c)=\sum_d f(\leftidx_ax_d\leftidx_d\beta_c)=\sum_d f(\leftidx_ax_d)\leftidx_d\beta_c.\] Thus it follows that $G(f)=g$ and that $G$ is full. Since $G$ is fully faithful and essentially surjective now it follows that $G$ is an equivalence of categories.

It suffices to show $G$ is monoidal. A monoidal functor $(G,\epsilon)$ has some extra structure than a functor. We require the data of a natural isomorphic $\epsilon_{-,-}$ and show it satisfies some criterion. First, by our definition of based bimodules we have that $(X,V_X) \otimes_{\KQ} (Y,V_Y)\cong (X \otimes_{\KQ} Y,V_X \otimes_{\KQ} V_Y)$. We compute the following, 

\[G((X \otimes_{\KQ}Y,V_X \otimes_{\KQ} V_Y))\cong(F_{V_X \otimes_{\KQ} V_Y}, \phi_{F_{V_{X \otimes_{\KQ} Y}}})\] 

\[G((X,V_X))\otimes G((Y,V_Y))\cong (F_{V_X},\phi_{F_X})\otimes (F_{V_Y},\phi_{F_Y}).\]

The tensor product of the maps (suppressing associators) is defined by the morphism \[\phi_{F_{X \otimes_{\KQ} Y}}= (\id_{F_{V_X}} \otimes \phi_{F_Y})\circ(\phi_{F_X} \otimes \id_{F_{V_Y}}).\] This is equivalent to the half braiding defined for $(F_{V_X},\phi_{F_X})\otimes (F_{V_Y},\phi_{F_Y})$.
Thus, it follows that $G((X \otimes_{\KQ}Y,V_X \otimes_{\KQ} V_Y)) \cong G((X,V_X))\otimes G((Y,V_Y))$ as objects in $\End_{(\widetilde{Q},m,i)}(\VecC(M))$. The isomorphism, $\epsilon_{X,Y}$ the morphism that identifies the composition of two functors with the functor they compose to and identifies the composition of two half braidings with the half braiding they compose to. By observation, this is a natural isomorphism.

 Now we need to check that our natural isomorphism $\epsilon_{X,Y}$ satisfies the diagram below.
In particular, when we plug in $G(X,V_X), G(Y,V_Y)$ and $G(Z,V_Z)$ we see that this diagram we need to check is,
   \[\begin{tikzcd}
	{(G(X,V_X)\otimes G(Y,V_Y))\otimes G(Z,V_Z)} && {G(X,V_X)\otimes (G(Y,V_Y)\otimes G(Z,V_Z))} \\
	G(X\otimes_{\KQ} Y,V_X\otimes_{\KQ} V_Y)\otimes G(Z,V_Z) && G(X,V_X) \otimes G(Y\otimes_{\KQ} Z,V_{Y\otimes_{\KQ} Z}) \\
  G\big{(}(X\otimes_{\KQ} Y)\otimes_{\KQ}Z,V_{(X\otimes_{\KQ} Y)\otimes_{\KQ}Z}\big{)}&& G\big{(}X\otimes_{\KQ} (Y\otimes_{\KQ}Z),V_{X\otimes_{\KQ} (Y\otimes_{\KQ}Z)}\big{)}
	\arrow["{\epsilon_{X,Y}\otimes_{\KQ}\id_Z}", from=1-1, to=2-1]
	\arrow["{\epsilon_{X\otimes_{\KQ} Y,Z}}", from=2-1, to=3-1]
	\arrow["{G(b_{X,Y,Z})}", from=3-1, to=3-3]
	\arrow["{a_{X,Y,Z}}", from=1-1, to=1-3]
	\arrow["{\id_X\otimes_A\epsilon_{Y,Z}}", from=1-3, to=2-3]
	\arrow["{\epsilon_{X,Y\otimes_{\KQ}Z}}", from=2-3, to=3-3]
\end{tikzcd}.\]
 By direct computation this diagram above commutes. 
\end{proof}

Thus, we conclude that $(G,\epsilon_{-,-})$ is an equivalence of monoidal categories.

\begin{cor}
   Given a monoidal functor $F:C\rightarrow \BasedBim(\KQ)_{I_v}$  there is a corresponding monoidal functor $H:C\rightarrow \End_{(\widetilde{Q},m,i)}(\VecC(M))$ and vice versa.
\end{cor} 
\begin{proof}
    This comes from the equivalence of monoidal categories. A monoidal functor $F:C\rightarrow \Bim(\KQ)$ must give us a monoidal functor $H:C\rightarrow \End_{(\widetilde{Q},m,i)}(\VecC(M))$ that factors through the monoidal equivalence $G$ stated in the above theorem.
\end{proof}

Now applying our theory of monoidal functors from $C$ to $\End_{(\widetilde{Q},m,i)}$ we can see the following two theorems are true.

\begin{thm}\label{thm:Actions of cat of path alg is moncat+dualfun}
   If a monoidal functor $F:C\rightarrow \BasedBim(\KQ)_{I_v}$ exists up to monoidal natural isomorphism it corresponds to,

   \begin{enumerate}
        \item A semisimple $C$ module category structure on $M$ up to invertible $C$ module functors of the form $(\Id,\rho)$.
        \item A conjugacy class of objects $[(\widetilde{Q},\phi_{-})]$ in $\End_C(\VecC(M))$ up to conjugaction by $C$ module endofunctors of the form $(\Id,\rho)$.
   \end{enumerate}
\end{thm}

\begin{proof}
    Follows from Theorem \protect\ref{thm:TE BimEnd}.
\end{proof}

  Based actions fully generalize graded and filtered actions on the path algebra. We now study graded and filtered based actions of $C$ on $\KQ$ since they correspond to properties of our one morphism $(\widetilde{Q},\phi_{-})$. 
    We define these two full subcategories for graded and filtered based bimodules of the path algebra.

\begin{define}
 We define the subcategory $\grBasedBim(\KQ)_{I_v}$ as the subcategory of idempotent split based bimodules who's half braiding respects the grading.
\end{define}

\begin{define}
 We define the subcategory $\filBasedBim(\KQ)_{I_v}$ as the subcategory of idempotent split based bimodules who's half braiding respects the filtration. 
\end{define}

\begin{remark}
    In both of these subcategories, the $V_X$ spaces are our zero graded/filtered component. 
\end{remark}
    
    We first define the endofunctor category $\End_Q(M)$. We will show that $\End_Q(M)$ equivalent to $\grBasedBim(\KQ)_{I_v}$ as monoidal categories.

    \begin{define}
        We define the category $\End_Q(M)=\{(F,\phi_F)|F\in \End(M), \phi_F:Q\circ F\xrightarrow{\sim}F\circ Q\}$ where morphisms are natural transformations that commute with the half braiding. 
    \end{define}

\begin{lem}\label{lem:graded actions are mod cat + dual functor}
  There is an equivalence of monoidal categories between $\grBasedBim(\KQ)_{I_v}$ and  
 $\End_{Q}(M)$, where $M$ is the semisimple category equivalent to $\Mod(\KQ_0)$.
\end{lem}

\begin{proof}
    Much of the proof is the same as Theorem \protect\ref{thm:TE BimEnd}. Thus, our functor is as follows, 

    \[H:\grBasedBim(\KQ)_{I_v}\rightarrow \End_Q(M).\]

    \[(X,V_X)\rightarrow (F_{V_X},\phi_{F_X}).\]

    Where $\phi_F$ is the natural isomorphism from $\phi_{F_X}: Q\circ F_X\xrightarrow{\sim} F_X\circ Q$.
\end{proof}

\begin{lem} \label{lem:graded actions are module cat and endo}
    Fix a fusion category $C$. Up to monoidal natural isomorphism, monoidal functors $F:C\rightarrow \End_{Q}(M)$ are classified by,

    \begin{enumerate}
        \item A semisimple $C$ module category structure on $M$ up to invertible $C$ module functors of the form $(\Id,\rho)$.
        \item A conjugacy class of objects $[(Q,\phi_{-})]$ in $\End_C(M)$ up to conjugation by 
 a $C$ module endofunctors of the form $(\Id,\rho)$. 
   \end{enumerate}

\end{lem}
   \begin{proof}
       The proof is the same as Theorem \protect\ref{thm:based action bijection}.
   \end{proof}

   \begin{cor}
   Fix an algebra complete fusion category $C$. Up to monoidal natural isomorphism, monoidal functors $F:C\rightarrow \End_{Q}(M)$ are classified by,

    \begin{enumerate}
        \item A choice of a family of bijections $\chi_i:\Irr(M_i)\rightarrow \Irr(M_i)$.
        \item An isomorphism class of objects $(Q,\phi_{-})$ in $\End_C(M)$.
   \end{enumerate}
\end{cor}
 \begin{proof}
       The proof is the same as Corollary \protect\ref{cor:algebra complete mon nat iso}.
   \end{proof}

   \begin{remark}
       With this connection we just estabilished, one could think of our origional monoidal category $\End_{\widetilde{Q}}(\VecC(M))$ as equivalent to graded actions of an infinite quivers $\widetilde{Q}$.
   \end{remark}

\begin{define}
    Define $\End_{(\Id \oplus Q,i)}(M)=\{(F,\phi_{F})|F\in \End(M),\hspace{1mm} \phi_{F}:(\Id\oplus Q) \circ F\xrightarrow{\sim}F\circ (\Id\oplus Q)\}$ such that $\phi_{F}\circ(i\otimes \id_F)=(\id_F \otimes i) \circ \phi^{\Id}_{F}$, where $i:\Id\rightarrow \Id\oplus Q$ is the embedding and $\phi_F^{\Id}$ is the trivial half braiding. 
    Morphisms are natural transformations that commute with the half braiding.
\end{define}

\begin{lem}
     There is an equivalence of monoidal categories between $\filBasedBim(\KQ)_{I_v}$ and  
 $\End_{(\Id \oplus Q,i)}(M)$, where $M$ is the semisimple category equivalent to $\Mod(\KQ_0)$.
\end{lem}
\begin{proof}
    The proof is essentially the same as the proof of Lemma \protect\ref{lem:graded actions are mod cat + dual functor} and Theorem \protect\ref{thm:TE BimEnd}.
\end{proof}

We produce a reduction of the filtered based bimodule case to the graded bimodule case. In the following lemma, we shall show that filtered actions are classified by the same data as graded actions. The idea used in this proof is a modification of the idea used in \protect\cite[Proposition 3.19]{EKW21}.

\begin{lem}\label{lem:fil actions are graded}
  Up to monoidal natural isomorphism, monoidal functors $F:C\rightarrow \End_{(Q\oplus\Id,i)}(M)$ are classified by
  \begin{enumerate}
        \item A semisimple $C$ module category structure on $M$ up to invertible $C$ module functors of the form $(\Id,\rho)$.
        \item A conjugacy class of objects $[(Q,\phi_{-})]$ in $\End_C(M)$ up to conjugation by 
 a $C$ module endofunctors of the form $(\Id,\rho)$. 
   \end{enumerate} 
\end{lem}
\begin{proof}
       A monoidal functor $F:C\rightarrow \End_{(\Id\oplus Q,i)}(M)$ corresponds to a semisimple module category $M$ and a dual functor $(\Id \oplus Q,\phi_{-})$ up to $C$ module endofunctors of the form $(\Id,\rho)$. Since $\End_C(M)$ is a multifusion category, it has a unit $(\Id,\psi_{-})$ with the trivial half braiding. In particular, $\End_C(M)$ is an abelian category so all cokernels exist. Thus, we produce the following short exact sequence,

       \[0\rightarrow (\Id,\psi_{-})\xrightarrow{i}(\Id\oplus Q,\phi_{-})\rightarrow (\Coker(i),\varphi_{-})\rightarrow 0.\]
       Since $\End_C(M)$ is a multifusion category this short exact sequence splits, which implies that $(\Coker(i),\varphi_{})$ is isomorphism as a $C$ module endofunctor to $(Q,\phi^Q)$. Thus, an action of $C$ of filtered based bimodules of $\KQ$ is classified by a semisimple $C$ module category and a $C$ module endofunctors take up to conjugation by a $C$ module functor of the form $(\Id,\rho)$.
\end{proof}
\begin{cor}
    All separable filtered actions of $C$ on $\KQ$ up to conjugacy are in fact separable graded actions of $C$ on $\KQ$ up to conjugacy.
\end{cor}

\begin{proof}
    This follows since separable  actions given us separable idempotent split actions. The half brading does not change when we recover the separable action, thus filtered separable actions of $C$ on $\KQ$ are graded.
\end{proof}

\begin{remark}
    The work done in Lemma \protect\ref{lem:graded actions are mod cat + dual functor} similar data to classify graded/filtered actions of fusion categories on path algebras that \protect\cite{EKW21} produced to classify finitely generated tensor algebras. The main difference, as before, is the fineness of our equivalence compared to the equivalence used in \protect\cite{EKW21}. The framework of $F:C\rightarrow \Bim(A)$ allows us to study actions of fusion categories on path algebras to this degree of fineness. One can refer to their paper for a plethora of useful examples and computations of module categories and module functors.
\end{remark}

\section{\texorpdfstring{\centering  Based Actions of Fusion Categories on $\KQ$}{Based Actions of Fusion Categories on Path Algebras}}
In this section, we will apply our theory to some examples. We will first start out by looking at graded actions. Then we will study based actions of $\PSU(2)_{p-2}$ on $\BasedBim(\KQ)_{I_v}$.
\subsection{\texorpdfstring{\centering Graded Actions of $C$ on $\KQ$}{ Graded Actions}}

In this subsection, we shall apply our theory of graded actions of fusion categories on path algebras to a few examples. 
We will first show some examples in $\VecC(G)$. To apply the theory developed in Lemma \protect\ref{lem:fil actions are graded} and Lemma \protect\ref{lem:graded actions are mod cat + dual functor} we need to understand what our $\VecC(G)$ module categories and their dual endofunctors.
\begin{define}
    The twisted group algebra $\K  H_\varphi$ in $\VecC(G)$ is $\bigoplus_{h\in H} \delta_h$ with multiplication $\delta_h \delta_{h'}=\varphi(h,h')\delta_{hh'}$ where $\varphi \in H^2(G,\K)$. 
\end{define}
We denote the category of right $\K H_\varphi$ modules by $M(H,\varphi)$. We can classify module categories of $\VecC(G)$ by the following data from \protect\cite{EGNO16}.
    Every indecomposable semisimple module category over $\VecC(G)$ is defined by a pair $(L,\psi)$ where $L \leq G$ and $\psi \in H^2(L,\K^*)$. Equivalence of module categories $(L,\psi)$ and $(L',\psi')$ occurs when $L'=gLg^{-1}$ and $\psi'$ is cohomologous to $\psi^g:=\psi(gxg^{-1},gyg^{-1})$.

Now that we understand the classification of semisimple module categories of $\VecC(G)$ we now need to understand the classification of dual functors between those module categories.
Dual functors are represented by $(L,\psi)- (K,\phi)$ bimodules plus some extra data. These bimodules are actually $L \times K^{op}$ equivariant objects. Thus, by \protect\cite{BN12} we have the following classification of $L \times K^{op}$ equivariant objects.

    Let $\{g_i\}$ be a choice of representative of the $L-K$ double cosets. Then $(L,\psi)- (K,\phi)$ bimodules can be classified by pairs $(g_k,\rho)$ where
    \begin{enumerate}
        \item $g_i$ represents a $L-K$ double coset.
        \item $\rho$ is an irreducible $\pi$ projective representation of the stabilizer $\Stab_{L\times K^{op}}(g_i)$ where $\pi=\psi\times \phi^{op}|_{\Stab_{L\times K^{op}}(g_i)}$.
    \end{enumerate}

\begin{example}
    Let's first look at the example of $\VecC(\Z/p\Z)$. The only semisimple indecomposable module categories of $\VecC(\Z/p\Z)$ are $\VecC(\Z/p\Z)$ and $\VecC_{f.d}$. Since $\Z/p\Z$ and the trivial group have trivial second cohomology, there are no twists of these module categories. These semisimple module categories corresponded to the algebras $
    \1$ and the group algebra $\K(\Z/p\Z)$.

    For $\Mod(\1)$, module endofunctors correspond to pairs of double cosets and choice projective representation of the stabilizer. Each element in $\VecC(\Z/p\Z)$ is in its own double cosets and the stabilizer is trivial, thus we have $p$ distinct module endofunctors corresponding to our double cosets with trivial stabilizers. 

    For $\Mod(\K(\Z/p\Z))$, there is one double coset of $\K(\Z/p\Z)$, namely itself. Then our choice of projective representation of the stabilizer is just any representation of $\Z/p\Z$ which there are $p$ of and all of them have trivial cohomology, leaving us with that there are $p$ distinct module endofunctors. More generally it has been proven that $\End_C(\Mod(\1))\cong \VecC(\Z/p\Z)^{op}$ and  $\End_{\VecC(\Z/p\Z)}(\Mod(\K\Z/p\Z)) \cong \Rep(\Z/p\Z)$ which confirms our explanation above. 
    
    If $M$ is not indecomposable then $M$ is a direct sum of $\VecC$ and $\VecC(\Z/p\Z)$ and our dual endofunctor can be broken down into block matrices with $\Hom_{\VecC(\Z/p\Z)}(M_i,M_j)$ in each block. Dual functors between $\VecC$ and $\VecC(\Z/p\Z)$ correspond to $e-\Z/p\Z$ double cosets which there is one corresponding to $\Z/p\Z$ which stabilizer equal to $\Z/p\Z$. Thus, there are $p$ module functors from $\VecC$ to $\VecC(\Z/p\Z)$. The computation is the same from $\VecC(\Z/p\Z)$ to $\VecC$. 
\end{example}

\begin{example}
    Now looking at $\VecC(\Z/n\Z)$. The exact algebras in $\VecC(\Z/n\Z)$ are group algebras $\K\Z/m\Z$ corresponding to the divisors of $n$. These algebras correspond to indecomposable semisimple module categories. Given a subgroup, $\Z/m\Z$ the $\VecC(\Z/n\Z)$ module endofunctors are classified by double cosets plus irreducible projective representation of the stabilizer. Since $\Z/n\Z$ is abelian double cosets are the same as left cosets. Thus, there are $\frac{n}{m}$ cosets. The stabilizer of each of these cosets is $\Z/m\Z$. Thus, there are $m$ irreducible representations of the stabilizer for each subgroup $\Z/m\Z$. 
    
    If our module category is decomposable, then we need to compute dual functors from $\Mod(\Z/m\Z)$ to $\Mod(\Z/k\Z)$. This corresponds to $\Z/m\Z-\Z/k\Z$ double cosets paired with a projective representation of their stabilizer. $\Z/m\Z-\Z/k\Z$ double cosets are $\Z/mk\Z$ left cosets whose stabilizers are equivalent to $\Z/(\frac{mk}{\gcd(m,k)})\Z$ which there are $\frac{mk}{\gcd(m,k)}$ irreducible representations of $\Z/(\frac{mk}{\gcd(m,k)})\Z$.
\end{example}

\begin{remark}
    For each of these last two examples, if we want to classify monoidal functor $F:\VecC(\Z_n)\rightarrow \grBasedBim(\KQ)_{I_v}$ we must be a little more specific. For each module category, we must choose a bijection of the isomorphism classes of simple elements in $M$ and an ordering of the indecomposable module categories. Once we choose that bijection, the structure of those module categories and $\VecC(\Z_n)$ module endofunctors will be described in the examples above.
    \end{remark}

For a more in depth classification of module categories and dual functors of pointed fusion categories, refer to \protect\cite[Section 4]{EKW21}. The authors delve into classification of $C$ module categories and $C$ module functors for $\VecC(G)$ with non-trivial second cohomology, $\VecC(G,\omega)$ for non-trivial $\omega$ and some group theoretic fusion categories.

With this is mind, we now understand what data defines a graded action of $C$ on $\KQ$. To classify all graded actions of $C$ on a path algebra is equivalent to classifying all the semisimple module categories and corresponding dual categories up to $C$ module functors of the form $(\Id,\rho)$. For algebra complete fusion categories, we can explicitly write down this data in terms of properties of a quiver $Q$.

\begin{thm}\label{thm: graded action classification}
    Fix a quiver $Q$ and an algebra complete fusion category $C$. Then there is a bijection between 
    
    \begin{enumerate}
        \item Separable idempotent split graded actions of $C$ on $\KQ$ up to conjugacy of based actions
        \item Partitions of the vertices into $n$ sets $S_1,...,S_n$ of size $|\Irr(C)|$ such that for each $S_k$ we pick a bijection $\psi_k:S_k\rightarrow \Irr(C)$ that maps each vertex $v_{l_k}$ to an isomorphism class of simple object $X \in \Irr(C)$ and the subquiver $Q_{ij}$ of all paths from $S_i$ to $S_j$ can be expressed as $|(Q_{ij})_{l_i,m_{j}}|=N_{Z_{ij},\psi_i(v_{l_i})}^{\psi_j(v_{m_j})}$ for some isomorphism class of objects $Z_{ij}\in C$.
    \end{enumerate}

\end{thm}

\begin{proof}
  By Lemma \ref{lem: conjugacy of actions} all idempotent split graded actions are conjugate to idempotent split actions over the canonical vertex projection. Lemma \protect\ref{lem:graded actions are module cat and endo} implies that an equivalence class of idempotent split graded actions over the vertex projections of $C$ on $\KQ$ is in bijection with a choice of a family of bijections $\chi_i:\Irr(M_i)\rightarrow \Irr(M_i)$ and a dual $(Q,\phi_{-})$ in $\End_C(M)$. Thus is suffices to show that this data is in bijection with the ordered partitions above.

 Choose a family of bijections $\chi_i:\Irr(C)\rightarrow \Irr(C)$. This is clearly in bijection with a partition $S_1,...,S_n$ of the vertices such that we pick a bijection $\psi_k$ for $S_k$.

 Now consider $C$ module endofunctor $(Q,\phi)$ in $\End_C(M)$. Since $(Q,\phi_{-})$ is a $C$ module endofunctor on $C^{\oplus n}$, it follows that $(Q,\phi_{-})$ can be decomposed into a bunch of objects $Q_{ij}$ in $\End_C(C)\cong C^{op}$. This implies that each object $Q_{ij}\cong Z_{ij}$ as elements in $C^{op}$. If we translate the data of $Q$ and in particular $Q_{ij}$ into quiver data it's in bijection  with quivers $Q$ such that $Q_{ij}$ of all paths from $S_i$ to $S_j$ can be expressed as $|(Q_{ij})_{l_i,l_{j}}|=N_{Z_{ij},\psi_i(v_{l_i})}^{\psi_j(v_{l_j})}$ for some isomorphism class of objects $Z_{ij}\in C$.
\end{proof}

\begin{example}
    Let's now apply this theory to the fusion category $\Fib$. Fix a quiver $Q$. As stated earlier, $\Fib$ is an algebra complete fusion category, which implies that the only semisimple indecomposable module category of $\Fib$ is $\Fib$. This implies that for $\Fib$, $\End_{\Fib}(\Fib)\cong \Fib^{op}$. By Lemma \protect\ref{lem:fil actions are graded} and semisimplicity, we can say that our idempotent split graded based actions correspond to a choice of a bijection of the simple objects and dual functor. That is $Q$ must be of the form,
    \[|Q_{ij}|=N_{Z,\psi(v_i)}^{\psi(v_j)} ,\]
for some $Z\cong \1^{\oplus n} \oplus \tau^{\oplus m}$.

If we let, $M \cong \Fib^{\oplus k}$ then the dual functor corresponds to k bijections of $\Irr(\Fib)\rightarrow \Irr(\Fib)$ and a choice of dual endofunctor. Thus our endofunctor must be of the form,
 \[|(Q_{ij})_{l_i,m_{j}}|=N_{Z_{ij},\psi_k(v_{l_i})}^{\psi_j(v_{m_j})} \]
for some $Z_{ij}\cong \1^{\oplus n_{ij}} \oplus \tau^{\oplus m_{ij}}$.
\end{example}

 It's a natural follow-up to ask if these graded actions are the only idempotent split based actions of $C$ on a path algebra $\KQ$. In the next example, we will construct a based action on $\BasedBim(\KQ)_{I_v}$ that does not preserve the grading or the filtration, giving us motivation to consider all based actions of $\BasedBim(\KQ)_{I_v}$.
\begin{example}
    Consider the following quiver $Q$,

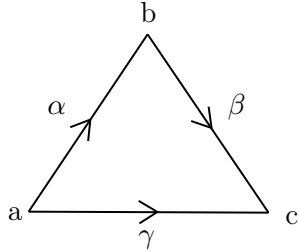
\begin{center}

\tikzset{every picture/.style={line width=0.75pt}} 

\begin{tikzpicture}[x=0.75pt,y=0.75pt,yscale=-1,xscale=1]

\draw    (320,90) -- (381,180) ;
\draw    (260,179.5) -- (320,90) ;
\draw    (260,179.5) -- (381,180) ;
\draw   (280.98,136.89) -- (291.14,133.02) -- (290.02,143.83) ;
\draw   (315.51,174.05) -- (324.94,179.64) -- (315.87,185.44) ;
\draw   (350.77,126.5) -- (352.26,137.26) -- (341.98,133.75) ;

\draw (248,176) node [anchor=north west][inner sep=0.75pt]   [align=left] {a};
\draw (315,72) node [anchor=north west][inner sep=0.75pt]   [align=left] {b};
\draw (388,178) node [anchor=north west][inner sep=0.75pt]   [align=left] {c};
\draw (268,122.4) node [anchor=north west][inner sep=0.75pt]    {$\alpha $};
\draw (359,118.4) node [anchor=north west][inner sep=0.75pt]    {$\beta $};
\draw (314,187.4) node [anchor=north west][inner sep=0.75pt]    {$\gamma $};

\end{tikzpicture}

\end{center}

The corresponding path algebra $\KQ$ over $Q$ is a $7$ dimensional algebra with basis elements $\{p_a,p_b,p_c,\alpha,\beta,\gamma,\beta\alpha\}$.
 Let $\phi:\KQ\rightarrow \KQ$ be the algebra automorphism that fixes each basis element except $\phi(\gamma)=-\gamma -\beta\alpha$. Notice that,

\[\phi^2(\gamma)=\phi(-\gamma -\beta\alpha)=-\phi(\gamma)-\phi(\beta\alpha)=-(-\gamma-\beta\alpha)-\beta\alpha=\gamma.\]
Thus, $\phi^2=\id$. 

Since this automorphism is not inner this induces a non-trivial action of $\Z_2$ on $\KQ$. Then by Lemma \protect\ref{lem:hopf action anives based action} we produce a based action of $\VecC(\Z_2)$ that does not preserve the filtration. In particular, $F(X_0):= \KQ \otimes X_0\cong \KQ$ and $F(X_1):=\KQ \otimes X_1$ with based spaces $\1_{\KQ}\otimes X_i$. The braiding of these bimodules is defined by the action of $g$ on $\KQ$. Then we apply our splitting of the base space into our idempotent split form giving us that bases spaces for each bimodules that are $V_{X_0}=\{p_a\otimes X_0, p_b\otimes X_0, p_c\otimes X_0\}$ and $V_{X_1}=\{p_a\otimes X_1, p_b\otimes X_1, p_c\otimes X_1\}$. This split doesn't change the braiding, and thus we produce a based action on the idempotent split bimodules such that the action is not graded or filtered. An example of this form motivates our interest to study based actions where the half braiding doesn't necessarily preserve the filtration/grading.
\end{example}

\subsection{\texorpdfstring{\centering Based Actions of $\PSU(2)_{p-2}$ on $\KQ$} {Based Actions of the PSU(2) on Path Algebras}}

We would like to find a family of fusion categories that our theory can fully classify up to equivalence separable idempotent split based actions on $\KQ$.
For most fusion categories, the existence of non-trivial semisimple module categories makes classification challenging. We shall show that the only idempotent split based actions of the family of fusion categories $\PSU(2)_{p-2}$ on $\KQ$ are actions coming from Lemma \protect\ref{lem:graded actions are mod cat + dual functor}, thus giving us a full classification up to conjugacy of based actions of $\PSU(2)_{p-2}$ on $\KQ$.

The family of fusion categories $\PSU(2)_{p-2}$ are a collection of algebra complete fusion categories with simple objects indexed by $\{X_{2j},0\leq j\leq \frac{p-3}{2}\}$ where $X_0\cong \1$, satisfying the fusion rules from \protect\cite{BK01},
\[X_{2j}\otimes X_{2i}=\sum N_{2i,2j}^{2m}V_{2m}.\]

Where,

\[N_{2i,2j}^{2m}=
\begin{cases}
    1 \hspace{1mm}\text{if} \hspace{1mm} |2i-2j|\leq 2m\leq \min \{2i+2j, 2(p-2)-2i-2j \} \\
    0 \hspace{1mm} \text{otherwise}
    
\end{cases}.\]

Given a simple object $X_{2j}$ , $\dim(X_{2j})=[2j+1]_q$ and unnormalized $S$-matrix entries are $S_{2j,2k}=[(2j+1)(2k+1)]_q$, where $q=e^{\frac{\pi i}{p}}$ a $2p^{th}$ root of unity. The Fibonacci Category, $\Fib$ discussed earlier, is a member of this family $\Fib \cong \PSU(2)_{5-2}$. By the fusion rules above, $\PSU(2)_3$ has two simple elements $X_0$ and $X_2$ satisfying $X_2 \otimes X_2 \cong X_0\oplus X_2$.

 $\PSU(2)_{p-2}$ is a family of modular tensor categories. A modular category is a non-degenerate ribbon fusion category \protect\cite{EGNO16,BK01}. In particular, the $S$ matrix diagonalizes the fusion rules of $C$. In particular for $\PSU(2)_{p-2}$ this implies that the columns of the $S$ matrix are a basis for $\R^{\frac{p-1}{2}}$. There have been some extensive results shown about orbits of Galois groups of modular tensor categories \protect\cite{NWZ,PSYZ24}. In particular, it was shown that $\PSU(2)_{p-2}$ is a transitive modular category \protect\cite{NWZ}. The transitivity implies that if we divide the $S_{-,2j}$ column of the $S$-matrix entries by $d_{X_{2j}}$, then the Galois group of $\Q(\frac{S_{X,Y}}{d_Y}|X,Y\in \Irr(\PSU(2)_{p-2})= \Q(d_X|X\in \Irr(\PSU(2)_{p-2})$ acts transitively on these quotient columns \protect\cite{NWZ}. Furthermore, $\Q(d_X|X\in \Irr(\PSU(2)_{p-2})$ is a real index $2$ subfield of $\Q(\omega_{2p})$ where $\omega_{2p}$ is a primitive $2p^{th}$ root of unity. We shall use this transitivity to show that all separable idempotent split based actions of $\PSU(2)_{p-2}$ on path algebras are in fact graded actions, implying that all separable based actions of $\PSU(2)_{p-2}$ on path algebras are in fact graded actions. To accomplish this, we will state and prove a few results about $\PSU(2)_{p-2}$.

\begin{lem} \label{lem: PSU2 Smat entries}
    The entries $S$-matrix of the fusion category $\PSU(2)_{p-2}$ can be represented in terms of $\pm d_X$ for some simple element $X$.
\end{lem}
\begin{proof}
     Let $X_{2l}$ be a simple object where $0\leq l\leq \frac{p-3}{2}$. The dimensions of the simple objects are $\dim(X_{2l})=[2l+1]_q$ where $q$ is a primitive $2p^{th}$ root of unity.

    \[\dim(X_{2l})=\frac{q^{2l+1}-q^{2p-(2l+1)}}{q-q^{2p-1}}=\frac{\sin(\frac{(2l+1)\pi}{p})}{\sin(\frac{\pi}{p})}.\]

    Since $l$ ranges from $0\leq l\leq \frac{p-3}{2}$ the quantum integers will stop at the odd number before $p$. This implies that the dimensions of the simple objects are $[1]_q,...,[p-2]_q$. Now consider $S_{2j,2k}$,

    \[S_{2j,2k}=\frac{q^{(2j+1)(2k+1)\modd 2p}-q^{2p-{(2j+1)(2k+1)}\modd 2p}}{q-q^{2p-1}}.\]

    Notice that since $(2k+1)$ and $(2j+1)$ are odd, then their product will be an odd $2p^{th}$ root of unity. In particular, $(2k+1),(2j+1)< p$ then $(2k+1)(2j+1)\neq 0\modd p$. Let $(2j+1)(2k+1)\modd 2p=2i+1$,

    \[S_{2j,2k}=\frac{q^{2i+1}-q^{2p-{(2i+1)}}}{q-q^{2p-1}}.\]

    Note $2i+1\neq 2p-(2i+1)$. If $0\leq 2i\leq p-3$ then $S_{2j,2k}=d(X_{2i})$ and if $p+1\leq 2i \leq 2p-2$ then $S_{2j,2k}=-d(X_{2p-2i-2})$.
\end{proof}

\begin{cor} \label{cor: PSU Smat cols}
    The $S$-matrix of the fusion category $\PSU(2)_{p-2}$ each row and column of the $S$ matrix has an entry $\pm d_X$ for all $X \in \Irr(\PSU(2)_{p-2})$.
\end{cor}
\begin{proof}
    Consider the characters associated with the $S$-matrix. These are functionals where we replace each entry $S_{a,b}$ by $\frac{S_{a,b}}{d_b}$. That is we have functions for each $Y \in \Irr(C)$, $\chi_Y:\Irr(C)\rightarrow \C$ define by,

    \[\chi_Y(X):=\frac{S_{X,Y}}{d_Y}.\]
    
    It follows by \protect\cite{NWZ} that the Galois group associated with $\Q(\frac{S_{a,b}}{d_b}|a,b\in \Irr(\PSU(2)_{p-2}))$ acts transitively on the characters. Suppose that two entries of a column of the $S$ matrix had $\pm d_X$ as an entry. This implies that for some $a\neq c,b\in \Irr(\PSU(2)_{p-2})$ that $S_{a,b}=\pm S_{c,b}$, then $\frac{S_{a,b}}{d_b}=\pm \frac{S_{c,b}}{d_b}$. Since the Galois action is transitive, there exists a $\sigma$ such that $\sigma(\frac{S_{a,0}}{d_0})=\frac{S_{a,b}}{d_b}$. This implies that $\sigma_b(d_a)=\pm\sigma_b(d_c)$. It follows that $\sigma_b(d_a)\mp\sigma_b(d_c)=0$. Since $\sigma$ is a field automorphism, $\sigma_b(d_a\mp d_c)=0$. This implies that, $d_a=\pm d_c$ which contradicts the rational independence of simple objects in $\PSU(2)_{p-2}$ \protect\cite[Lemma 3.3]{EJ24}. Thus, for each column $\pm d_X$ appears exactly once and since $S$ is symmetric this property applies for rows as well.
 \end{proof}

  We are going to need some more information about the decomposition of tensor products of objects in $\PSU(2)_{p-2}$ into simple objects. It's a well known nontrivial result that $\Z[\omega_n]$ is the ring of integers for the field $\Q(\omega_{n})$ for any root of unity $\omega_{n}$ \protect\cite[Theorem 6.4]{Mil08}. In addition it has been shown that the ring of integers of $\Q(\omega_p+\omega_p^{-1})$, the largest real subfield of $\Q(\omega_p)$, is $\Z[\omega_p+\omega_p^{-1}]$. Since $\Q(\omega_p)=\Q(\omega_{2p})$ it implies that the ring of integers $\Z[\omega_p]=\Z[\omega_{2p}]$ as rings. There is sizable effort to understand the units of the $\Z[\omega_p]$ and the corresponding real subring $\Z[\omega_p+\omega_p^{-1}]$. There is a family of units in $\Z[\omega_p+\omega_p^{-1}]$ called the real cyclotomic units defined as the following real numbers,

  \[\omega_p^{\frac{1-a}{2}}\frac{{1-\omega_p^a}}{1-\omega_p}=\frac{\omega_p^{\frac{a}{2}}-\omega_p^{\frac{-a}{2}}}{\omega_p^{\frac{1}{2}}-\omega_p^{\frac{-1}{2}}}=\frac{\sin(\frac{a\pi}{p})}{\sin(\frac{\pi}{p})}, \hspace{1mm}1<a<\frac{p}{2}.\]

  These units are well studied and the group generated by them is closely connected to the class number of (real) cyclotomic rings \protect\cite{Mil14,Was12,DDK19}. We notice that there are $\frac{p-1}{2}-1$ of these units and that they correspond to the dimensions of our simple objects in $\PSU(2)_{p-2}$. The dimensions of our simple objects are,

  \[d_{X_{2k}}=\frac{\sin(\frac{(2k+1)\pi}{p})}{\sin(\frac{\pi}{p})}.\]

  If $1<(2k+1)<\frac{p}{2}$ then $d_{X_{2k}}$ is a real cyclotomic unit, otherwise if $\frac{p}{2}<(2k+1)<p-1$ then $\sin(\frac{(2k+1)\pi}{p})=\sin(\frac{(p-(2k+1))\pi}{p})$ which is an even number $a$ such that $1<a<\frac{p}{2}$ implying that $d_{X_{2k}}$ for $2k+1>\frac{p}{2}$ is also a real cyclotomic unit. We also note that there are both $\frac{p-1}{2}-1$ of these cyclotomic units and dimensions $d_{X_{2k}}$. Thus the real cyclotomic units are in bijection with the $d_{X_{2j}}$ for a prime $p$.

  The group of real cyclotomic units is the group generated by $\langle -1,d_{X_{2k}}|1\leq k\leq \frac{p-2}{2}\rangle$. By \protect\cite[Lemma 8.1]{Was12} these numbers are multiplicatively independent and thus the group generated by them is isomorphic as groups to $\Z/2\Z\oplus \Z^{\frac{p-3}{2}}$. Thus we have arrived at the following corollary.

\begin{cor}\label{cor:uniqueness of tensor product in PSU(2)}
Let $Y$ be an isomorphism class of objects in $\PSU(2)_{p-2}$. If $Y$ has a decomposition in the tensor product of simple objects, then that decomposition is unique up to a sequence of braidings and multiplications by the unit.
\end{cor}
\begin{proof}
    Follows from \protect\cite[Lemma 8.1] {Was12}. If $Y$ has two  decompositions into non-trivial simple object then the those have to correspond to the same element in the group of real cyclotomic units in $\Z[\omega_p+\omega_p^{-1}]$. Since these objects are multiplicativly independent then $\dim(Y)$ decomposes into a product of positive powers of the generators, which is unique with the exception of multiplication by $1$.
    Therefore it follows that there decompostion is the same up to half braidings and multiplication by the unit.
\end{proof}
Let $Q$ be a finite quiver, then $\widetilde{Q}$ is the quiver generated by $Q$. We are interested in classifying monoidal functors from $\PSU(2)_{p-2}$ to $\BasedBim(\KQ)_{I_v}$. By Theorem \protect\ref{thm:TE BimEnd} this is equivalent to a semisimple module category and a dual functor $\widetilde{Q}$ satisfying some properties. The dual functor decomposes into sums of simple dual functors in $\End_C(M)$.

We have already constructed a nice family of separable idempotent split actions on $\PSU(2)_{p-2}$ coming from graded actions up to conjugacy. If fix a quiver $Q$ a graded separable idempotent split based action up to equivalence are in bijection with partitions of the vertices into $n$ sets $S_1,...,S_n$ of size $\frac{p-1}{2}$ such that for each $S_k$ we pick a bijection $\psi_k:S_k\rightarrow \Irr(\PSU(2)_{p-2}$ that maps each vertex $v_{l_k}$ to an isomorphism class of simple object $X \in \Irr(C)$ and the subquiver $Q_{ij}$ of all paths from $S_i$ to $S_j$ can be expressed as $|(Q_{ij})_{l_i,m_{j}}|=N_{Z_{ij},\psi_i(v_{l_i})}^{\psi_j(v_{m_j})}$ for some isomorphism class of objects $Z_{ij}\in C$. We are going to show that these are up to equivalence the only separable idempotent split based actions of $\PSU(2)_{p-2}$ on $\KQ$. We will first prove a couple small propositions that will make proving our main results easier.  Define $T_{X_2}$ to be the linear operator define by conjugation by the fusion matrix of $X_2$. The fusion matrix associated with the simple element $X_2$ in $\PSU(2)_{p-2}$ with respect to the order basis $\{X_0,...,X_{2\frac{p-3}{2}}\}$ is a $(\frac{p-1}{2})\times (\frac{p-1}{2})$ matrix with the entries,

\[F_{X_2}=\begin{bmatrix}
    0 & 1 & 0  & \dots & 0&0\\
     1 & 1 & 1 & 0 & \dots&0\\
     0 & 1  & 1  & \ddots &0 &\vdots\\
     \vdots & 0&1 &\ddots& 1 & 0\\
     0 & \vdots &0 & \ddots &1 &1\\
     0 & 0 & \dots& 0 & 1& 1
\end{bmatrix}.\]

This is an invertible symmetric matrix implying that the matrix transformation arising from conjugation by $X_2$ is $T_{X_2}=F_{X_2}^{T}\otimes F_{X_2}^{-1}=F_{X_2}\otimes F_{X_2}^{-1}$. Since $F_{X_2}$ is diagonalizable then it follows that $F_{X_2}^{-1}$ is also diagonalizable and it follows that all eigenpairs can be found by taking Kronecker products of eigenpairs of $F_{X_2}$ and $F_{X_2}^{-1}$. That is if $(\alpha_i,v_i)$ and $(\beta_j,u_j)$ are our eigenpairs of $F_{X_2}$, then the eigenpairs of the Kronecker product are $(\alpha_i\otimes \beta_j^{-1},u_i\otimes v_j)_{i,j}$.

\begin{prop} \label{prop:no eigenvector -1}
   $T_{X_2}$ has no eigenvectors of eigenvalue $-1$
\end{prop}

\begin{proof}
     Given an isomorphism class of simple objects, $X
_j$ the $S$ matrix diagonalizes the fusion rules $X_{2i}$. The corresponding eigenvalues are, 

\[(\lambda_{i})_j=\frac{S_{i,j}}{d_{X_j}}.\]

Notice that $\sigma((\lambda_{i})_j)=\sigma(\frac{S_{i,j}}{d_{X_j}})=\frac{S_{i,\sigma(j)}}{d_{X_{\sigma(j)}}}$ for a Galois automorphism $\sigma$. This implies that Galois group acts on the eigenvalues of $X_{2i}$. For $\PSU(2)_{p-2}$, choosing our simple object as $X_2$ we see that,

\[(\lambda_{X_2})_{2j}=\frac{S_{2,2j}}{d_{X_{2j}}}.\]

It suffices to show that $(\lambda_Y)_{2j}\neq -(\lambda_Y)_{2k}$ for all $k\neq j$. Suppose that

\[\frac{S_{2,2j}}{d_{X_{2j}}}+\frac{S_{2,2k}}{d_{X_{2k}}}=0 \] 
for some $k\neq j$. This implies that,

\[S_{2,2j}d_{X_{2k}}=S_{2,2k}d_{X_{2j}}.\]

By Lemma \protect\ref{lem: PSU2 Smat entries} it follows that $S$-matrix entries can be expressed as $\pm d_{X_{2i}}$ for some simple object $X_{2i}$. Thus,

\[\pm d_{X_{2i}}d_{X_{2k}}=\pm d_{X_{2l}}d_{X_{2j}}.\]

By Corollary \protect\ref{cor:uniqueness of tensor product in PSU(2)}, since $j\neq k$ it follows that $d_{X_{2i}}d_{X_{2k}}=d_{X_{2l}}d_{X_{2j}}$ only when $X_{2i}\cong X_{2j}$ and $X_{2l}\cong X_{2k}$. Thus, it follows that,

\[\frac{S_{2,2j}}{d_{X_{2j}}}=\frac{\pm d_{X_{2j}}}{d_{X_{2j}}}=\pm 1, \frac{S_{2,2k}}{d_{X_{2k}}}=\frac{\pm d_{X_{2k}}}{d_{X_{2k}}}=\pm 1.\]

This implies that the fusion matrix of $X_2$ has an eigenvector of eigenvalue $1$. This is a contraction, since the Galois group acts transitively on the eigenvalues of $X_2$, thus none of the eigenvalues can be rational. Therefore if $a$ is an eigenvalue then $-a$ cannot be an eigenvalue.
\end{proof}

An immediate realization from the proof is that eigenvalues of $X_{2k}$ are unique for any $1\leq k\leq \frac{p-1}{2}$. This implies that the diagonal matrix has distinct entries on each diagonal for any non-trivial simple object.

\begin{prop} \label{prop:unique magnitudes}
    The magnitude of the eigenvalues of $T_{X_2}$ that are not equal to $1$ are distinct.
\end{prop}
\begin{proof}
    Suppose that there are two eigenvalues of with the same magnitude. This implies,

    \[\left|\frac{S_{2,2j}}{d_{X_{2j}}}\frac{d_{X_{2k}}}{S_{2,2k}}\right|=\left|\frac{S_{2,2l}}{d_{X_{2l}}}\frac{d_{X_{2m}}}{S_{2,2m}}\right|\]
    Then this implies,
 \[\frac{S_{2,2j}}{d_{X_{2j}}}\frac{d_{X_{2k}}}{S_{2,2k}}\pm\frac{S_{2,2l}}{d_{X_{2l}}}\frac{d_{X_{2m}}}{S_{2,2m}}=0\]
 Finding a common denominator,

 \[\frac{S_{2,2j}d_{X_{2k}}d_{X_{2l}}S_{2,2m}\pm d_{X_{2j}}S_{2,2l}S_{2,2k}d_{X_{2m}}}{d_{X_{2j}}S_{2,2k}d_{X_{2l}}S_{2,2m}}=0\]

 Which implies,

 \[S_{2,2j}d_{X_{2k}}d_{X_{2l}}S_{2,2m}\pm d_{X_{2j}}S_{2,2l}S_{2,2k}d_{X_{2m}}=0\]

 By Lemma \protect\ref{lem: PSU2 Smat entries}, it follows that each $S_{2,2a}$ is equal to $d_{X_{2b}}$ for some isomorphism class of simple objects $X_{2b}$. Corollary \protect\ref{cor:uniqueness of tensor product in PSU(2)} implies that the dimensions of both sides on both sides of the equation must be the same up to swapping. This implies $S_{2,2j}$ must equal one of $d_{X_{2j}},S_{2,2l},S_{2,2k},d_{X_{2m}}$. If $S_{2,2j}=\pm d_{X_{2j}}$ this is a contradiction, since then we have an eigenvector of eigenvalue $1$ contradiction the transitivity of our Galois action. If $S_{2,2j}=S_{2,2k}$ it implies that,

 \[\left|\frac{S_{2,2j}}{d_{X_{2j}}}\frac{d_{X_{2k}}}{S_{2,2k}}\right|=1.\]
 This gives us an eigenvalue $1$ in $T_{X_2}$ which is allowed. If $S_{2,2j}=S_{2,2l}$ then it follows that $l=j$ and then it implies that $m=k$ for the desired equation to be true. We note in this case that by the proof of Proposition \protect\ref{prop:no eigenvector -1} if $a$ is an eigenvalue then $-a$ cannot be an eigenvalue as well. Thus, the only potential problematic case is when $S_{2,2j}=\pm d_{X_{2m}}$. By a similar arguement, it follows that the only nonobvious case is when $S_{2,2k}=d_{X_{2l}}$, $S_{2,2m}=d_{X_{2j}}$, and $S_{2,2l}=d_{X_{2k}}$. It follows that we have reduced our equation to,
 \[\left|\frac{d_{X_{2m}}}{d_{X_{2j}}}\frac{d_{X_{2k}}}{d_{X_{2l}}}\right|=\left|\frac{d_{X_{2k}}}{d_{X_{2l}}}\frac{d_{X_{2m}}}{d_{X_{2j}}}\right|.\]

Let $a=\frac{d_{X_{2m}}}{d_{X_{2j}}}$ and $b=\frac{d_{X_{2k}}}{d_{X_{2l}}}$. By definition, one of, $\pm a$ and $\pm b$ are eigenvalues for the fusion matrix of $X_2$. For this equation to be true, it implies that one of $\pm \frac{1}{a}$ and $\pm \frac{1}{b}$ are eigenvalues of the fusion matrix of $X_2$. Since the Galois group acts transitively on the eigenvalues, it follows that there exists automorphism $\sigma$ and $\gamma$ such that $\sigma(a)=\pm \frac{1}{a}$ and $\gamma(b)=\pm \frac{1}{b}$. This implies that $a\pm \frac{1}{a}$ is fixed by the automorphism $\sigma$, thus is in $\Q$. Consider the equation $a\pm \frac{1}{a}=\frac{r}{s}$. Solutions to this equation are $a$ such that $a^2+\frac{r}{s}a\pm1=0$ which implies that $a$ lives in a degree two extension of $\Q$. The only degree two extension of $\Q(\omega_p)$ is either $\Q(\sqrt{p})$ or $\Q(\sqrt{ip})$. If $\frac{p-1}{2}$ is odd then it follows that there is no quadratic subfeild of $Q(\omega_p+\omega_p^{-1})$ implying that this case cannot happen. Thus suppose that $\frac{p-1}{2}$ is even and $a \in \Q(\sqrt{p})$.  Now since the Galois group acts transitively on the eigenvalues it follows that there exists an automorphism $\rho$ such that $\rho(a)=d_{X_2}$, thus $\rho(a\pm\frac{1}{a})=d_{X_2}\pm\frac{1}{d_{X_2}}=\frac{r}{s}$. This produces the equation,

\[d_{X_2}^2-\frac{r}{s}d_{X_2}\pm 1=0.\]

For $p\neq 5$, $d_{X_2}^2=d_{X_4}+d_{X_2}+1$, thus

\[sd_{X_4}+sd_{X_2}+s-rd_{X_2}\pm s=0.\]

It follows that by the rational independence of $d_{X_{2j}}$ \protect\cite{EJ24} that there is no possible solution unless $p=5$ and and $s=r$ and $\pm 1=-1$. This is precisely the case when $\PSU(2)_{3}\cong \Fib$ and the eigenvalues are $\frac{1}{d_{X_2}}$ and $-d_{X_2}$, in this case we can check that there are unique  eigenvalues magnitudes by observation. For $p\neq 5$ it follows that this senario cannot happen as well and the result follows.

\end{proof}

\begin{prop}\label{prop: S matrix neg/pos}
    Let $S_{X_{2j}}$ be a column of the $S$ matrix of $\PSU(2)_{p-2}$ then if $j\neq 0$ then $S_{X_{2j}}$ has a negative and positive entry.
\end{prop}

\begin{proof}
Since $\PSU(2)_{p-2}$ is a transitive modular category that the $S$ matrix is a real symmetric matrix and each simple object is self dual \protect\cite[Lemma 3.3]{NWZ}. Since $S_{X,0}=d_X$ then it follows that column corresponding to the identity $X_0$ in the $S$-matrix is strictly positive. Since $\PSU(2)_{p-2}$ is transitive, it follows that the columns of the $S$ matrix are orthogonal \protect\cite{NWZ}, implying that the dot product of $S_{X_i}\cdot S_{X_j}=0$ unless $i=j$. Thus, $S_{X_0}  \cdot S_{X_{2j}}=0$ for all $j\neq 0$ which implies that all the other columns have a negative entry. $S_{0,X}=d_X$ so there is also a positive entry in each column. An entry cannot be zero, otherwise it would contradict the transitivity of the Galois action.

\end{proof}

This immediatly implies that $S_{X_{2i}}\otimes S_{X_{2j}}$ will have a negative and positive entry unless $i=j=0$. If $i\neq j$ then at least one $S_{X_{2i}},S_{X_{2j}}$ has a negative entry, $-d_{2_i}$. Then $-d_{2i}S_{0,2j}$ will be negative and $S_{0,2i}S_{0,2j}$ will be positive. The arguement where $S_{X_{2j}}$ has the negative entry is analogous. We now use all of these propositions to describe what quivers can have a separable based action on them.

\begin{thm}\label{thm:PSU(2) actions}
 Let $Q$ be a quiver with $\frac{p-1}{2}$ vertices. If there exists a monoidal functor $F:\PSU(2)_{p-2}\rightarrow \BasedBim(\KQ)_{I_v}$ then there is a choice of ordered basis for $Q$ such that the adjacency matrix of $Q$ is of the form, \[|Q_{ij}|=N_{Z,X_{2(i-1)}}^{X_{2(j-1)}} ,\]
 for some isomorphism class of objects $Z$ in $\PSU(2)_{p-2}$.

\end{thm}
\begin{proof}
A  based action $\PSU(2)_{p-2}$ on $\BasedBim(\KQ)_{I_v}$ is equivalent to a monoidal functor $F:C\rightarrow \End_{(\widetilde{Q},m,i)}(\VecC(M))$. This assigns a bijection $\psi$ on the simple objects $\{Y_1,...,Y_{\frac{p-3}{2}}\}$. The quiver $Q$ can be realized as a finite subfunctor of $\widetilde{Q}$. Since $(F_X,\phi_{F_X})$ is an object in $\End_{(\widetilde{Q},m,i)}(\VecC(M))$ it follows that $\phi_{F_X}$ is an isomorphism. $\widetilde{Q}\circ F_X\xrightarrow{\phi_{F_X}} F_X\circ \widetilde{Q}$. Consider the restriction to the subfunctor $Q$, $Q\circ F_X\xrightarrow{\phi_{F_X}} F_X\circ \hat{Q}$, where $Q$ is an adjacency matrix of a quiver. In particular, $\hat{Q}$ must also be a non-negative integer matrix. We want to show that the only matrices that remain nonnegative integer matrices under finite braidings are the matrices of the form,
$|Q_{ij}|=N_{Z,X_{2i}}^{X_{2j}}$ for some isomorphism class of objects $Z$ in $\PSU(2)_{p-2}$.

We will pick our simple object $X$ to be $X_2$ and consider $F_{X_2}^{-1}QF_{X_2}=\hat{Q}$. The fusion matrix associated with the simple element $X_2$ in $\PSU(2)_{p-2}$ with respect to the order basis $\{X_0,...,X_{2\frac{p-3}{2}}\}$ is a $(\frac{p-1}{2})\times (\frac{p-1}{2})$ matrix with the entries,

\[F_{X_2}=\begin{bmatrix}
    0 & 1 & 0  & \dots & 0&0\\
     1 & 1 & 1 & 0 & \dots&0\\
     0 & 1  & 1  & \ddots &0 &\vdots\\
     \vdots & 0&1 &\ddots& 1 & 0\\
     0 & \vdots &0 & \ddots &1 &1\\
     0 & 0 & \dots& 0 & 1& 1
\end{bmatrix}.\]

Given this, we want to show that a non-negative integer matrix $Q$ such that $T^k_{X_2}(Q)$ is also a non-negative integer matrix for all $k$ then $Q$ must be of the form, $|Q_{ij}|=N_{Z,X_{2i}}^{X_{2j}}$  for some isomorphism class of objects $Z$ in $\PSU(2)_{p-2}$.

Since $\PSU(2)_{p-2}$ is a modular tensor category it follows that the $S$-matrix diagonalizes each simple object. This implies that the columns of the $S$ matrix are the eigenvectors and the eigenvalues can be computed by the formula in \protect\cite{BK01},

\[(D_i)_{XZ}=\delta_{XZ}\frac{S_{i,X}}{S_{\1,X}}.\]

There are three properties of the operator $T_{X_2}$ we need to show to prove this result. 
\begin{enumerate}
\item All eigenvectors have eigenvalues $|\lambda|> 1, \lambda=1,$ or $1>|\lambda|$. In particular, there is no eigenvalue $-1$.

\item  The magnitude of the eigenvalues of $T_{X_2}$ that are not equal to $1$ are distinct.
  \item Each eigenvalues $|\lambda|>1$ the corresponding eigenvector has a negative and a positive entry in it.

\end{enumerate}
These three conditions follow from Propositions \protect\ref{prop:no eigenvector -1}, \protect\ref{prop:unique magnitudes}, \protect\ref{prop: S matrix neg/pos}.

The eigenpairs $(\lambda_{ij},S_{ij})$ of conjugation by $F_{X_2}$ are,

\[\lambda_{ij}=(\frac{S_{2,2i}}{d_{X_{2i}}})(\frac{S_{2,2j}}{d_{X_{2j}}})^{-1},\hspace{1mm} S_{ij}=S_i\otimes S_j.\] 
Since $F_{X_2}$ is diagonalizable, it follows that $F_{X_2}\otimes F_{X_2}^{-1}$ is diagonalizable, and the eigenvectors $\{S_{ij}\}$ form a basis of $\R^{(\frac{p-1}{2})^2}$. Thus, any potential $Q$ we would like to consider must be a linear combination of these eigenvectors. If $Q =\sum \alpha_{ij}S_{ij}$ then applying our conjugation transformation $T_{F_{X_2}}$ will scale $\alpha_{ij}S_{ij}$ by $\lambda_{ij}$. Since the  conjugation operator is equivalent to the Kronecker product, there will be $\frac{p-1}{2}$ eigenvectors with eigenvalue $1$, and then $\frac{(\frac{p-1}{2})^2-(\frac{p-1}{2})}{2}$ eigenvectors with eigenvalues $|\lambda|<1$ and $|\lambda|>|1|$. Since this is a finite list, consider the one with the largest magnitude eigenvector, denote it $\mu_{kl}$ with the corresponding eigenvector $S_k\otimes S_l$. By Proposition \protect\ref{prop:unique magnitudes} $\mu_{kl}$ is unique. Since $\mu_{kl}\neq 1$, it follows that $S_k\otimes S_l$ contains a negative and a positive entry. The sign of the entries of $T_{X_2}^k(Q)$ as $k\rightarrow \infty$ are completely determined by the sign of the entries of $S_k\otimes S_l$ since $\mu_{ij}$ is the largest magnitude eigenvalue. Thus, $S_k\otimes S_l$ cannot be in the decomposition of $T_{X_2}^k(Q)$ otherwise we would guarantee the existence of a negative entry for some $k$, which is impossible. Repeating this for each eigenvalue $\mu$ such that $|\mu|>1$ eliminates all the corresponding eigenvectors from the decomposition. 

Now we consider the eigenvectors with $|\lambda|<1$ and $\lambda=1$. Let $Q=\sum \beta_{mn}S_{mn}^1+ \alpha_{ij}S_{ij}^{<1}$ where $S_{mn}^1$ is an eigenvector of eigenvalue 1 and $S_{ij}^{<1}$ have eigenvalue $|\lambda_{ij}|<1$. We require that $T_{X_2}^k(Q)$ is a non-negative integer matrix for each $k\in \N$.  There are two cases to consider, when $\sum \beta_{mn}S_{mn}^1 \in \Mat(\Z)$ and when $\sum \beta_{mn}S_{mn}^1\not \in \Mat(\Z)$. When $\sum \beta_{mn}S_{mn}^1\not \in \Mat(\Z)$ then there exists at least one entry in the column vector that is not an integer. Then this entry is $\epsilon$ away from the nearest integer for some $\epsilon>0$. We can choose a $K$ such that for each $ij$,

\[|\alpha_{ij}\lambda^K_{ij}S_{ij}^{<1}|<\left(\frac{2\epsilon}{(\frac{p-1}{2})^2-(\frac{p-1}{2})},\frac{2\epsilon}{(\frac{p-1}{2})^2-(\frac{p-1}{2})},\dots,\frac{2\epsilon}{(\frac{p-1}{2})^2-(\frac{p-1}{2})}\right)^T.\]

It follows that,
\[T^K_{X_2}(Q)=\sum \beta_{mn}S_{mn}^1+ \alpha_{ij}\lambda^K_{ij}S_{ij}^{<1} \not \in \Mat(\Z).\]

Suppose that for $\sum \beta_{mn}S_{mn}^1 \in \Mat(\Z)$. For some $K$ large enough, it follows that  

\[\sum \alpha_{ij}\lambda^K_{ij}S_{ij}^{<1} =0.\]

 But, this contradicts the linear independence of the basis vectors $S_{ij}^{<1}$, thus $\alpha_{ij}=0$ for all $ij$. Therefore, only the eigenvalue $1$ vectors can be in the decomposition. Notice that a basis for $S_{kl}^1$ eigenvectors of eigenvalue $1$ is $\{F_{X_{2j}}\}_{j=0}^{\frac{p-3}{2}}$. Since $F_{X_{2j}}$ is the only functor to take $\1$ to $X_{2j}$ it follows that the coefficient $n_{X_{2j}}\in \Z_{\geq 0}$, then it follows that, there is a choice of ordered basis for $Q$ such that the adjacency matrix of $Q$ is of the form, \[|Q_{ij}|=N_{Z,X_{2(i-1)}}^{X_{2(j-1)}} ,\]
 for some isomorphism class of objects $Z$ in $\PSU(2)_{p-2}$.
\end{proof}

\begin{cor} \label{cor:PSU(2) actions}
     If there exists a monoidal functor $F:\PSU(2)_{p-2}\rightarrow \BasedBim(\KQ)_{I_v}$ then there exists an ordered basis of $Q$ such that the adjacency matrix of $Q$ can be decomposed into a $k|\Irr(\PSU(2)_{p-2})|\times k|\Irr(\PSU(2)_{p-2})|$ block matrix where each block $Q_{ij}$ can be expressed in form,
     \[|(Q_{ij})_{lm}|=N_{Z_{ij},X_{2(l-1)}}^{X_{2(m-1)}}\]
     for some isomorphism class of objects $Z_{ij}$ in $\PSU(2)_{p-2}$.
\end{cor}

\begin{proof}
    Given a separable idempotent split based action of $\PSU(2)_{p-2}$ we can represent our quiver as a $k(\frac{p-1}{2})\times k(\frac{p-1}{2})$ block matrix. We conjugate by the block matrix with $X_2$ block on the diagonal.

     \[\begin{bmatrix}
            F_{X_2}&0&\dots& 0 \\
            0&F_{X_2}&\dots& 0 \\
            \vdots &\vdots &\ddots &\vdots  \\
            0 &0 & 0 &F_{X_2}
     \end{bmatrix} ^{-n}\begin{bmatrix}
            Q_{11}&Q_{12}& \dots& Q_{1n} \\
            Q_{21}&Q_{22}&\dots& Q_{2n} \\
            \vdots & \vdots  &\vdots & \vdots \\
            Q_{n1} &Q_{n2}  &\dots &Q_{nn}
     \end{bmatrix}\begin{bmatrix}
            F_{X_2}&0&\dots& 0 \\
            0&F_{X_2}&\dots& 0 \\
            \vdots &\vdots &\ddots &\vdots  \\
            0 &0 & 0 &F_{X_2}
     \end{bmatrix}^n . \]
Is equal to,
     \[\begin{bmatrix}
            F^{-n}_{X_2}Q_{11}F_{X_2}^n&F^{-n}_{X_2}Q_{12}F_{X_2}^n& \dots&F^{-n}_{X_2} Q_{1n}F_{X_2}^n \\
            F^{-n}_{X_2}Q_{21}F_{X_2}^n&F^{-n}_{X_2}Q_{22}F_{X_2}^n&\dots&F^{-1}_{X_2} Q_{2n}F_{X_2}^n \\
            \vdots & \vdots  &\vdots & \vdots \\
           F^{-n}_{X_2} Q_{n1}F_{X_2}^n &F^{-n}_{X_2}Q_{n2}F_{X_2}^n  &\dots &F^{-n}_{X_2}Q_{nn}F_{X_2}^n
     \end{bmatrix}.\]

     Then, by Theorem \protect\ref{thm:PSU(2) actions} it follows that this is an action only when $Q_{ij}$ is of the form,
     \[|(Q_{ij})_{lm}|=N_{Z_{ij},X_{2(l-1)}}^{X_{2(m-1)}},\]
for some isomorphism class of objects $Z_{ij}$ in $\PSU(2)_{p-2}$.
\end{proof}

\begin{example}
Let $p=7$ and consider $\PSU(2)_{7-2}=\PSU(2)_5$.
There are three simple objects in this category $X_0,X_2$ and $X_4$ with the dimensions $1$, $[3]_q$ and $[5]_q$ where $q=e^{\frac{\pi i}{7}}$ a $14^{th}$ root of unity.

The fusion matrix associated with $X_0,X_2$ and $X_4$ are,

\[X_2=\begin{bmatrix}
    1 &0 & 0 \\
    0 &1 & 0 \\
    0 &0 & 1 \\
\end{bmatrix},\hspace{1mm} X_2=\begin{bmatrix}
    0 &1 & 0 \\
    1 &1 & 1 \\
    0 &1 & 1 \\
\end{bmatrix},\hspace{1mm} X_4=\begin{bmatrix}
    0 &0 & 1 \\
    0 &1 & 1 \\
    1 &1 & 0 \\
\end{bmatrix}.\]

The unnormalized $S$-matrix for $\PSU(2)_5$ is,

\[\begin{bmatrix}
    1 &[3]_q & [5]_q \\
    [3]_q &[9]_q & [15]_q \\
    [5]_q &[15]_q & [25]_q \\
\end{bmatrix}.\]

Taking this $\modd 14$ yields,
\[\begin{bmatrix}
    [1]_q &[3]_q & [5]_q \\
    [3]_q &[9]_q & [1]_q \\
    [5]_q &[1]_q & [11]_q \\
\end{bmatrix}.\]

Notice that,

\[[9]_q= \frac{(e^{\frac{\pi i}{7}})^9-(e^{\frac{\pi i}{7}})^{-9}}{(e^{\frac{\pi i}{7}})-(e^{\frac{\pi i}{7}})^{-1}}.\]

This simplifies to,

\[[9]_q= \frac{(e^{\frac{\pi i}{7}})^9-(e^{\frac{\pi i}{7}})^{5}}{(e^{\frac{\pi i}{7}})-(e^{\frac{\pi i}{7}})^{13}}=-\frac{(e^{\frac{\pi i}{7}})^5-(e^{\frac{\pi i}{7}})^{9}}{(e^{\frac{\pi i}{7}})-(e^{\frac{\pi i}{7}})^{13}}=-[5]_q.\]

Similarly, we see that, $[11]_q=-[3]_q$ which implies that the $S$ matrix is,

\[\begin{bmatrix}
    [1]_q &[3]_q & [5]_q \\
    [3]_q &-[5]_q & [1]_q \\
    [5]_q &[1]_q & -[3]_q \\
\end{bmatrix}.\]

The corresponding eigenvalues are $\frac{[3]_q}{[1]_q}\simeq 2.245$, $\frac{-[5]_q}{[3]_q}\simeq -.802$ and $\frac{[1]_q}{[5]_q}\simeq .555$. The corresponding eigenpairs for the conjugation action are,

\[\begin{bmatrix}
    \hspace{5mm}1  & \hspace{9.5mm}\frac{2.245} {-.802}&\hspace{5.5mm} \frac{2.245}{.555} &\hspace{5.5mm}  \frac{-.802}{2.245} & \hspace{8mm} 1 & \hspace{9.5mm} \frac{-.802}{.555} & \hspace{6mm} \frac{.555}{2.245} & \hspace{5.5mm} \frac{.555}{-.802} & \hspace{10.5mm}  1 \phantom{blah}
\end{bmatrix}.\]

\[\begin{bmatrix}
    [1]_q [1]_q &[1]_q[3]_q & [1]_q[5]_q  & [3]_q[1]_q &   [3]_q[3]_q& [3]_q[5]_q &[5]_q[1]_q   & [5]_q[3]_q  & [5]_q[5]_q \\
    [1]_q[3]_q& -[1]_q[5]_q&    [1]_q[1]_q & [3]_q[3]_q& -[3]_q[5]_q &[3]_q[1]_q&[5]_q[3]_q   & -[5]_q[5]_q  & [5]_q[1]_q \\
    [1]_q[5]_q & [1]_q[1]_q& -[1]_q[3]_q  &  [3]_q[5]_q& [3]_q[1]_q & -[3]_q[3]_q& [5]_q[5]_q  &  [5]_q[1]_q  &  -[5]_q[3]_q\\
     [3]_q[1]_q & [3]_q[3]_q&    [3]_q[5]_q &  -[5]_q[1]_q&-[5]_q[3]_q&-[5]_q[5]_q& [1]_q[1]_q  & [1]_q[3]_q  &  [1]_q[5]_q\\
    [3]_q[3]_q & -[3]_q[5]_q&   [3]_q[1]_q &  -[5]_q[3]_q& [5]_q[5]_q&-[5]_q[1]_q&  [1]_q[3]_q  & -[1]_q[5]_q & [1]_q[1]_q\\
    [3]_q[5]_q & [3]_q[1]_q&     -[3]_q[3]_q & -[5]_q[5]_q&-[5]_q[1]_q&[5]]_q[3]_q& [1]_q[5]_q   & [1]_q[1]_q  &-[1]_q[3]_q \\
     [5]_q[1]_q & [5]_q[3]_q&   [5]_q[5]_q  &  [1]_q[1]_q& [1]_q[3]_q&[1]_q[5]_q&   -[3]_q[1]_q  &  -[3]_q[3]_q &-[3]_q[5]_q\\
    [5]_q[3]_q & -[5]_q[5]_q&   [5]_q[1]_q   &  [1]_q[3]_q& -[1]_q[5]_q&[1]_q[1]_q&  -[3]_q[3]_q  &   [3]_q[5]_q & -[3]_q[1]_q\\
    [5]_q[5]_q & [5]_q[1]_q&    -[5]_q[3]_q & [1]_q[5]_q&[1]_q[1]_q& -[1]_q[3]_q& -[3]_q[5]_q  &  [3]_q[1]_q & [3]_q[3]_q\\ 
\end{bmatrix}.\]

Then we apply the results from our theorem to eliminate any $Q$ except those that are positive integer combinations of the eigenvectors of eigenvalue $1$. Those eigenvectors have a nice basis consisting of $X_0,X_2,X_4$.

Thus, we conclude that with respect to this choice of simple object basis, if $F:\PSU(2)_5\rightarrow \BasedBim(\KQ)_{I_v}$ then $Q$ has a choice of ordered basis such that $Q$ is $3k \times 3k$ block matrix where each block is of the form,

\[Q_{ij}=\begin{bmatrix}
    n_0 & n_2 & n_4 \\
     n_2 & n_0+n_2+n_4 & n_2+n_4 \\
      n_4 & n_2+n_4 & n_0+n_2 \\
\end{bmatrix}\]
\end{example}

Now we will to write down a bijection between separable idempotent split based actions of $\PSU(2)_{p-2}$ on $\KQ$ and some quiver data. This will give us a full classification of the monoidal functors of $\PSU(2)_{p-2}$ on $\BasedBim(\KQ)_{I_v}$ up to monoidal natural isomorphism. Fix a quiver $Q$. This generates a quiver $\widetilde{Q}$. Then, given $F:C\rightarrow \End_{\widetilde{Q}}(\VecC(M))$ ($M\cong C^{\oplus n}$) up to monoidal natural isomorphism $F$ corresponds to,
\begin{enumerate}
        \item A choice of family of bijections $\chi_i:\Irr(C)\rightarrow \Irr(C)$.
        \item A dual functor $(\widetilde{Q},\phi_{-})$ in $\End_C(\VecC(M))$. 
   \end{enumerate} 

Consequently, there is only one isomorphism $\phi_X:\widetilde{Q}\circ F_X\xrightarrow{\sim}F_X\circ \widetilde{Q}$. Since $\widetilde{Q}$ is generated by $Q$, if there were two different half braidings of $Q\circ F_X\xrightarrow{\phi_1} F_X\circ \hat{Q}$ and $Q\circ F_X\xrightarrow{\phi_2} F_X\circ \overline{Q}$ this would contradict our maximum of one action. In Theorem \protect\ref{thm: graded action classification} the quiver data for those actions corresponds to the only separable idempotent split based action $\PSU(2)_{p-2}$ of $\KQ$ with respect to those choice of bijections and quiver. In particular, for each ordering of a quiver that satisfies the criterion of Corollary \protect\ref{cor:PSU(2) actions}, we have constructed a graded separable idempotent split action in Theorem \protect\ref{thm: graded action classification}. Furthermore, this graded action is the only separable idempotent split action up to conjugacy of $\PSU(2)_{p-2}$ on this partition by the reasoning above. Thus, we have in fact proved the following, 
\begin{thm}\label{thm: classification of actions of PSU}
   Every separable idempotent split based action of $\PSU(2)_{p-2}$ on $\KQ$ up to conjugacy is a graded separable idempotent split action of $\PSU(2)_{p-2}$ on $\KQ$. 
\end{thm}

With this in mind, we state our full classification theorem for based actions of $\PSU(2)_{p-2}$ on $\BasedBim(\KQ)_{I_v}$ up to conjugacy
of based actions. 

\begin{thm} \label{thm:classification PSU}
Fix a quiver $Q$. Then there is a bijection between,
    
\begin{enumerate}
        \item Separable idempotent split based actions of $\PSU(2)_{p-2}$ up to conjugacy of based actions 
        \item Partitions of the vertices into $n$ sets $S_1,...,S_n$ of size $\frac{p-1}{2}$ such that for each $S_k$ we pick a bijection $\psi_k:S_k\rightarrow \Irr(\PSU(2)_{p-2}$ that maps each vertex $v_{l_k}$ to an isomorphism class of simple object $X \in \Irr(C)$ and the subquiver $Q_{ij}$ of all paths from $S_i$ to $S_j$ can be expressed as $|(Q_{ij})_{l_i,m_{j}}|=N_{Z_{ij},\psi_i(v_{l_i})}^{\psi_j(v_{m_j})}$ for some isomorphism class of objects $Z_{ij}\in C$.
    \end{enumerate} 
    \end{thm}

\begin{proof}
   By Corollary \protect\ref{cor:PSU(2) actions} it follows that there exists a based action of $\PSU(2)_{p-2}$ on $\BasedBim(\KQ)_{I_v}$ then it $Q$ must have the desired decomposition. For each partition decomposition, we get exactly one separable idempotent split graded action of $\PSU(2)_{p-2}$ on $\KQ$ up to conjuacy. Now by Theorem \protect\ref{thm: graded action classification} the result follows.
\end{proof}

Finally, we reference Theorem \protect\ref{thm:separable action and idempotent} to produce the following statement about separable based actions of $\PSU(2)_{p-2}$ on $\KQ$

\begin{thm}
    Fix a quiver $Q$. If there exists  separable based action of $\PSU(2)_{p-2}$ on $\KQ$ up to conjugacy then there exists a partition of the vertices into $n$ sets $S_1,...,S_n$ of size $\frac{p-1}{2}$ such that for each $S_k$ we pick a bijection $\psi_k:S_k\rightarrow \Irr(\PSU(2)_{p-2}$ that maps each vertex $v_{l_k}$ to an isomorphism class of simple object $X \in \Irr(C)$ and the subquiver $Q_{ij}$ of all paths from $S_i$ to $S_j$ can be expressed as $|(Q_{ij})_{l_i,m_{j}}|=N_{Z_{ij},\psi_i(v_{l_i})}^{\psi_j(v_{m_j})}$ for some isomorphism class of objects $Z_{ij}\in C$.
\end{thm}
\begin{proof}
    Follows from Theorem \protect\ref{thm:separable action and idempotent} and Theorem \protect\ref{thm:classification PSU}
\end{proof}

\begin{cor}
 Every seperable based action of $\PSU(2)_{p-2}$ on $\KQ$ up to conjugacy is a graded action of $\PSU(2)_{p-2}$ on $\KQ$. 
\end{cor}
\begin{proof}
    The decomposing of the action into the idemponent split case does not affect the half braiding of the object. Thus, since the half braiding in the split case preserves the grading, the half braiding in the general separable case preserves the grading.
\end{proof}

\section*{Acknowledgements}
The author would like to thank Corey Jones for introducing this project idea. His thoughtful knowledge, advice, and words of encouragement throughout made this possible. This work was partially supported by NSF DMS 2247202.

\bibliographystyle{alpha}
\bibliography{main}

\end{document}